\documentclass[12pt,twoside]{amsart}
\usepackage{amsmath}
\usepackage{amsthm}
\usepackage{amsfonts}
\usepackage{amssymb}
\usepackage{latexsym}
\usepackage[all]{xy}
\usepackage{epsfig}
\usepackage{amscd,mathrsfs,graphicx,mathrsfs}
\usepackage[numbers,sort&compress]{natbib}

\date{}
\allowdisplaybreaks[4] \footskip=15pt

\renewcommand{\uppercasenonmath}[1]{}

\numberwithin{equation}{section} \theoremstyle{plain}
\newtheorem*{theorem*}{Theorem A}
\newtheorem*{theorem**}{Theorem B}
\newtheorem{theorem}{Theorem}[section]
\newtheorem{corollary}[theorem]{Corollary}
\newtheorem*{corollary*}{Corollary}
\newtheorem{lemma}[theorem]{Lemma}
\newtheorem*{lemma*}{Lemma}
\newtheorem{proposition}[theorem]{Proposition}
\newtheorem*{proposition*}{Proposition}
\newtheorem{remark}[theorem]{Remark}
\newtheorem*{remark*}{Remark}
\newtheorem{example}[theorem]{Example}
\newtheorem*{example*}{Example}
\newtheorem{definition}[theorem]{Definition}
\newtheorem*{definition*}{Definition}

\newtheorem*{conjecture*}{Conjecture}
\newtheorem*{ack*}{ACKNOWLEDGEMENTS}

\newcommand{\pf}{\noindent\begin {proof}}
\newcommand{\epf}{\end{proof}}

\begin{document}
\begin{center}
{\Large \bf  Frobenius functors, stable equivalences and $K$-theory of Gorenstein projective modules
 \footnotetext{E-mail address: wren@cqnu.edu.cn.}}

\vspace{0.5cm}  Wei Ren\\
{\small   School of Mathematical Sciences, Chongqing Normal University, Chongqing 401331, P.R. China
}
\end{center}


\bigskip
\centerline { \bf  Abstract}
\leftskip10truemm \rightskip10truemm
Owing to the difference in $K$-theory, an example by Dugger and Shipley implies that the equivalence of stable categories of Gorenstein projective modules should not be a Quillen equivalence. We give a sufficient and necessary condition for the Frobenius pair of faithful functors between two abelian categories to be a Quillen equivalence, which is also equivalent to that the Frobenius functors induce mutually inverse equivalences between stable categories of Gorenstein projective objects.

We show that the category of Gorenstein projective objects is a Waldhausen category, then Gorenstein $K$-groups are introduced and characterized. As applications, we show that stable equivalences of Morita type preserve Gorenstein $K$-groups, CM-finiteness and CM-freeness. Two specific examples of path algebras are presented to illustrate the results, for which the Gorenstein $K_0$ and $K_1$-groups are calculated.
\bigskip


\leftskip0truemm \rightskip0truemm \vbox to 0.2cm{}

\section{\bf Introduction}

The origin of Gorenstein projective modules may date back to the study of $G$-dimension by Auslander and Bridger \cite{AB69} in 1960s. Under mild conditions, Gorenstein projective modules are also known as totally reflexive modules and maximal Cohen-Macaulay modules.  By \cite[Theorem 3.7]{DEH18}, the category formed by modules of finite Gorenstein projective dimension, which is a weakly idempotent complete exact category, admits a Gorenstein projective model structure. Recall that a model category, introduced by Quillen \cite{Qui67}, refers to a category with three specified classes of morphisms, called fibrations, cofibrations and weak equivalences, which satisfy a few axioms that are deliberately reminiscent of properties of topological spaces. The homotopy category associated to a model category is obtained by formally inverting the weak equivalences, while the objects are the same. For the sake of convenience and completeness, we briefly recall some notions on model categories and Quillen equivalences in Appendix \ref{sec:ModCat}, and refer to \cite{Qui67, Hov99, DS95} for more details.

There is an advantage that any Quillen equivalence between stable model categories induces a weak equivalence of $K$-theory spaces, and then yields isomorphic $K$-groups; see \cite[Corollary 3.10]{DS04}. However, Dugger and Shipley constructed in \cite{DS09} an example of triangulated-equivalent model categories which are not Quillen equivalent. For a fixed prime $p$, let $R = \mathbb{Z}/p^2$ and $R_{\epsilon}= k[\epsilon]/(\epsilon^2)$, where $k= \mathbb{Z}/p$. There are model structures $\mathcal{M}_{R}$ and $\mathcal{M}_{R_{\epsilon}}$ for the stable categories, such that the homotopy categories $\mathrm{Ho}(\mathcal{M}_{R})$ and $\mathrm{Ho}(\mathcal{M}_{R_{\epsilon}})$ are equivalent as triangulated categories. Schlichting \cite{Sch02} observed that when $p > 3$, the $K$-theory of the finitely-generated objects in each category differs starting at $K_{4}$: $K_{4}(\mathcal{M}_{R})\cong \mathbb{Z}/p^{2}$, whereas $K_{4}(\mathcal{M}_{R_{\epsilon}})\cong \mathbb{Z}/p \oplus \mathbb{Z}/p$. Hence, the difference in $K$-theory implies that $\mathcal{M}_{R}$ and $\mathcal{M}_{R_{\epsilon}}$ are not Quillen equivalent.

Note that every module over $R$ or $R_{\epsilon}$ is Gorenstein projective, and then there are equivalences $\mathrm{Ho}(\mathcal{M}_{R})\simeq \underline{\mathcal{GP}}(R)$ and $\mathrm{Ho}(\mathcal{M}_{R_{\epsilon}})\simeq \underline{\mathcal{GP}}(R_{\epsilon})$. Thus, in general we should not expect that equivalences of Gorenstein stable categories $\underline{\mathcal{GP}}$ of two rings are induced by Quillen equivalences with respect to Gorenstein projective model structures. A question raised naturally: when the equivalence between stable categories of Gorenstein projective modules is induced from a Quillen equivalence?

Our first main result gives an affirmative answer to the above question, see Theorem \ref{thm:FQE}. We give a sufficient and necessary condition for the Frobenius pair $(F, H)$ of faithful functors to be a Quillen equivalence with respect to Gorenstein projective model structures, which is also equivalent to that $(F, H)$ is the mutual inverse between $\underline{\mathcal{GP}}(\mathcal{A})$ and $\underline{\mathcal{GP}}(\mathcal{B})$. A typical example of Frobenius functor is the induction functor ${\rm Ind}_{H}^{G} = RG\otimes_{RH}-$ for group rings, where $H\subseteq G$ is a subgroup with finite index. Faithful Frobenius functors arise naturally in stable equivalences of Morita type, and in singular equivalences of Morita type; see for examples \cite{DM07, LX07, Xi08, ZZ13}.

Furthermore, we intend to study the $K$-theory with respect to Gorenstein projective objects. We first show in Lemma \ref{lem:Wad-GP} that the category of Gorenstein projective objects $\mathcal{GP}(\mathcal{A})$ is a Waldhausen category. Then we can apply Waldhausen's construction to define the notion of Gorenstein $K$-groups (Definition \ref{def:GK-group}). Especially, for finitely generated Gorenstein projective modules over a finite dimensional algebra $A$, Gorenstein $K_0$-group of $A$ is the Grothendieck group of isomorphism classes in $\underline{\mathcal{GP}}(A)$, or equivalently weak equivalence classes of $\mathcal{GP}(A)$; Gorenstein $K_1$-group of $A$ is generated by automorphisms in $\underline{\mathcal{GP}}(A)$, or equivalently weak self-equivalence classes of maps in $\mathcal{GP}(A)$; see Definition \ref{def:K0} and \ref{def:K1}. Recall that the classical $K_0$-group (Grothendieck group) and $K_1$-group (Bass-Whitehead group) of a ring are defined via finitely generated projective modules in the category of modules. Roughly speaking, we get a Gorenstein version of $K_0$ and $K_1$-groups defined in the stable category $\underline{\mathcal{GP}}(A)$ via finitely generated Gorenstein projective modules.

Recall that an algebra $A$ is CM-finite \cite{Bel05}, provided that there are only finitely many isomorphism classes of indecomposable finitely generated Gorenstein projective $A$-modules; $A$ is CM-free if all Gorenstein projective $A$-modules are projective. If $A$ is an Artin algebra, then $A$ is CM-finite if and only if there exists a finitely generated Gorenstein projective module $G$ such that $\mathcal{GP}(A) = \mathrm{add}G$.

There is an interesting relation between Gorenstein  $K_1$-groups and classical $K_1$-groups. Let $A$ be a CM-finite Artin algebra with $\mathcal{GP}(A) = \mathrm{add}G$, and let $\Lambda = \underline{\mathrm{End}}(G)$ be the endomorphism ring of $G$ in the stable category $\underline{\mathcal{GP}}(A)$. We prove in Theorem \ref{thm:CharK1} that Gorenstein $K_1$-group of $A$ is isomorphic to the classical $K_1$-group of the stable endomorphism ring $\Lambda$.

As applications, in Proposition \ref{prop:AdjType} we show that stable equivalences of Morita type yield the Frobenius pair of faithful functors which satisfies the equivalent conditions in Theorem \ref{thm:FQE}. Furthermore, it follows that Gorenstein $K$-groups, CM-finiteness and CM-freeness are invariant under stable equivalences of Morita type; see Corollary \ref{cor:invAdjType}. Following these invariant properties, Gorenstein $K_0$ and $K_1$-groups of some specific path algebras are explicitly calculated; see Proposition \ref{prop:calK0}, \ref{prop:calK1} and \ref{prop:invEg2}. It is worth noting that it is unknown whether stable equivalences of Morita type preserve classical $K$-groups, although classical $K$-groups are Morita invariants.

The structure of this paper is as follows. In Section 2, we recall notions and facts on Gorenstein projective model structure. The Quillen equivalences and stable equivalences induced by Frobenius functors are proved in Section 3. Gorenstein $K$-groups are studied in Section 4, which are invariant under Quillen equivalences. As applications, we show in Section 5 that some invariants, such as Gorenstein $K$-groups, CM-finiteness and CM-freeness, are preserved by stable equivalences of Morita type. Two specific examples are presented in Section 6. In the Appendix A and B, we briefly recall Waldhausen's construction of simplicial categories and the $K$-theory spaces, as well as Quillen model categories, respectively.

\section{\bf Preliminaries}

Let $\mathcal{A}$ be a complete and co-complete abelian category with enough projective objects. Recall that an acyclic complex of projective objects
$$\mathbf{P} = \cdots\longrightarrow P_{n+1}\longrightarrow P_{n}\longrightarrow P_{n-1}\longrightarrow\cdots$$
is said to be \emph{totally acyclic}, provided it remains acyclic after applying $\mathrm{Hom}_{\mathcal{A}}(-, P)$ for any projective object $P\in \mathcal{A}$. An object $M\in \mathcal{A}$ is \emph{Gorenstein projective} \cite{EJ00} if it is isomorphic to a cycle of such a totally acyclic complex of projectives, that is, if $M\cong Z_i(\mathbf{P})$ for some integer $i$. Denote by $\mathcal{GP}(\mathcal{A})$ and $\mathcal{P}(\mathcal{A})$ the full subcategories of $\mathcal{A}$ consisting of Gorenstein projective objects and projective objects, respectively. It is clear that $\mathcal{P}(\mathcal{A})\subseteq \mathcal{GP}(\mathcal{A})$.

Let $X$ be an object of $\mathcal{A}$. The projective dimension of $X$ is denoted by $\mathrm{pd}_\mathcal{A}(X)$. The \emph{Gorenstein projective dimension} $\mathrm{Gpd}_\mathcal{A}(X)$ of $X$ is defined by declaring that
$\mathrm{Gpd}_\mathcal{A}(X)\leq n$ if and only if there is an exact sequence $0\rightarrow G_{n}\rightarrow G_{n-1} \rightarrow \cdots \rightarrow G_{0}\rightarrow X\rightarrow 0$ with each $G_{i}\in \mathcal{GP}(\mathcal{A})$.

Recall that a \emph{Frobenius category} is an exact category $\mathcal{E}$ with enough projectives and enough injectives and where the classes of projectives and injectives coincide. The stable category $\underline{\mathcal{E}}$
is the triangulated category $\mathcal{E}/\sim$ where the relation ``$\sim$'' is defined by $f\sim g$ provided that $f-g$ factors through a projective-injective object.

It is well known that $\mathcal{GP}(\mathcal{A})$ is a Frobenius category, where the projective-injective objects are precisely projective objects in $\mathcal{A}$; see for example \cite[Proposition 2.2]{DEH18}. Thus, the stable category $\underline{\mathcal{GP}}(\mathcal{A})$ is triangulated. We will see in Proposition \ref{prop:GPModel} that $\underline{\mathcal{GP}}(\mathcal{A})$ can also be realized by homotopy category of a Gorenstein projective model category.

The following make sense when $\mathcal{A}$ is an exact category (not necessarily abelian). Assume that the exact category $\mathcal{A}$ has enough projectives and enough injectives. The Yoneda Ext bifunctor is denoted by $\mathrm{Ext}_{\mathcal{A}}(-, -)$. A pair of classes $(\mathcal{X}, \mathcal{Y})$ in $\mathcal{A}$ is a \emph{cotorsion pair} provided that $\mathcal{X} =  {^\perp}\mathcal{Y}$ and $\mathcal{Y} = \mathcal{X}^{\perp}$, where the left orthogonal class $^{\perp}\mathcal{Y}$ consists of $X$ such that $\mathrm{Ext}^{\geq 1}_{\mathcal{A}}(X, Y) = 0$ for all $Y\in \mathcal{Y}$, and the right orthogonal class $\mathcal{X}^{\perp}$ is defined similarly. The cotorsion pair $(\mathcal{X}, \mathcal{Y})$ is said to be {\em complete} if for any object $M\in \mathcal{A}$, there exist short exact sequences $0\rightarrow Y\rightarrow X \rightarrow M \rightarrow 0$ and $0\rightarrow M\rightarrow Y' \rightarrow X' \rightarrow 0$ with $X, X'\in \mathcal{X}$ and $Y, Y'\in \mathcal{Y}$. In this case, $X \rightarrow M$ is called a \emph{special $\mathcal{X}$-precover} (or, \emph{special right $\mathcal{X}$-approximation}) of $M$, and $M\rightarrow Y'$ is called  a \emph{special $\mathcal{Y}$-preenvelope} (or, \emph{special left $\mathcal{Y}$-approximation}) of $M$.

We recall some basic notions of model categories in Appendix B, and refer to \cite{Qui67, Hov99, DS95} for more details. Let $\mathcal{A}$ be a {\em weakly idempotent complete exact category} in the sense of
\cite[Definition 2.2]{Gil11}. By \cite[Theorem 3.3]{Gil11} there is a correspondence between complete cotorsion pairs and model structure, that is, $\mathcal{A}$ admits a model structure if and only if there is a triple $(\mathcal{A}_{c}, \mathcal{A}_{tri}, \mathcal{A}_{f})$ of subcategories, for which $\mathcal{A}_{tri}$ is thick in the sense of \cite[Definition 3.2]{Gil11}, and both $(\mathcal{A}_{c}, \mathcal{A}_{f}\cap \mathcal{A}_{tri})$ and $(\mathcal{A}_{c}\cap \mathcal{A}_{tri}, \mathcal{A}_{f})$ are complete cotorsion pairs. Then, the corresponding model structure can be denoted by the triple $(\mathcal{A}_{c}, \mathcal{A}_{tri}, \mathcal{A}_{f})$. We refer to \cite[Theorem 2.2]{Hov02} or \cite[Chapter VIII]{BR07} for the similar correspondence in abelian categories.

For a model category $\mathcal{A}$, its homotopy category $\mathrm{Ho}(\mathcal{A})$ is the localization of $\mathcal{A}$ with respect to the collection of weak equivalences. The subcategory $\mathcal{A}_{cf} = \mathcal{A}_{c}\cap\mathcal{A}_{f}$ of cofibrant-fibrant objects is a Frobenius category, with $\omega = \mathcal{A}_{c}\cap\mathcal{A}_{tri}\cap\mathcal{A}_{f}$ being the class of projective-injective objects. Then the stable category $\underline{\mathcal{A}_{cf}} = \mathcal{A}_{cf}/\omega$ is a triangulated category. Moreover, we have an equivalence of triangulate categories $\mathrm{Ho}(\mathcal{A})\simeq \underline{\mathcal{A}_{cf}}$ by
\cite[Theorem 1.2.10]{Hov99}.

Let $\mathcal{A}$ be an abelian category with enough projective objects. We use $\mathcal{GP}^{<\infty}(\mathcal{A})$ and  $\mathcal{P}^{<\infty}(\mathcal{A})$ to denote the subcategories of objects in $\mathcal{A}$ with finite Gorenstein projective dimension and finite projective dimension, respectively. Note that $\mathcal{GP}^{<\infty}(\mathcal{A})$ is a weakly idempotent complete exact category \cite[Lemma 3.3]{DEH18}, and $\mathcal{GP}(\mathcal{A})\cap \mathcal{P}^{<\infty}(\mathcal{A}) = \mathcal{P}(\mathcal{A})$ holds by \cite[Proposition 10.2.3]{EJ00}.

Recall that an abelian category $\mathcal{A}$ with enough projective and injective objects is called a \emph{Gorenstein category} \cite[Definition VII 2.1]{BR07} if both the supremum of projective dimension of all injective objects $\mathrm{spli}(\mathcal{A})$, and the supremum of injective dimension of all projective objects $\mathrm{silp}(\mathcal{A})$, are finite; compare with \cite[Definition 2.18]{EEGR}. This is equivalent to that $\mathcal{A}$ is of finite global Gorenstein projective dimension, that is, $\mathcal{A} = \mathcal{GP}^{<\infty}(\mathcal{A})$); see for example \cite[Proposition VII 1.3]{BR07} or \cite[Theorem A.6]{CR20}. If the cateogry of $R$-modules is a Gorenstein category, then $R$ is called a \emph{Gorenstein regular} ring (also named left/right-Gorenstein ring in \cite[Definition VII 2.5]{BR07}). Every \emph{Iwanaga-Gorenstein ring} (i.e. two-sided noetherian ring with finite self-injective dimension on both sides) is Gorenstein regular, while the converse may not be true in general.

Over an Iwanaga-Gorenstein ring, there is a Gorenstein projective model structure on the category of modules; see \cite[Theorem 8.6]{Hov02}. We recently generalize this result to Gorenstein regular rings; see
\cite[Theorem 2.7]{Ren18}. However, in order to obtain the Gorenstein projective model structure, it is not necessary to restrict the base ring $R$ if one only considers the weakly idempotent complete exact category $\mathcal{GP}^{<\infty}(R)$. In fact, $R$ is a Gorenstein regular ring if and only if $\mathcal{GP}^{<\infty}(R)$ equals to the category of all $R$-modules.

\begin{proposition}\label{prop:GPModel}(\cite[Theorem 3.7]{DEH18})
On the category $\mathcal{GP}^{<\infty}(\mathcal{A})$, there is a model structure $$(\mathcal{GP}(\mathcal{A}), \mathcal{P}^{<\infty}(\mathcal{A}), \mathcal{GP}^{<\infty}(\mathcal{A})).$$
In this case, cofibrations and trivial cofibrations are monomorphisms with cokernel in $\mathcal{GP}(\mathcal{A})$ and $\mathcal{P}(\mathcal{A})$, respectively; every epimorphism is a fibration, and trivial fibrations are epimorphisms with kernel in  $\mathcal{P}^{<\infty}(\mathcal{A})$; weak equivalences are morphisms which factor as a trivial cofibration followed by a trivial fibration.

Moreover, the homotopy category of this model category is equivalent, as a triangulated category, to the stable category of Gorenstein projective objects in $\mathcal{A}$; symbolically denoted with
$$\mathrm{Ho}(\mathcal{GP}^{<\infty}(\mathcal{A})) \simeq \underline{\mathcal{GP}}(\mathcal{A}).$$
\end{proposition}

\section{\bf Frobenius functors and stable categories of Gorenstein projective objects}

Throughout this section, we assume that both $\mathcal{A}$ and $\mathcal{B}$ are abelian categories with enough projective objects. Let $F: \mathcal{A}\rightarrow \mathcal{B}$ and $H: \mathcal{B}\rightarrow \mathcal{A}$ be two additive functors. We say that $(F, H)$ is a \emph{Frobenius pair} between $\mathcal{A}$ and $\mathcal{B}$,  provided that both $(F, H)$ and $(H, F)$ are adjoint pairs. If $F$ fits into a Frobenius pair $(F, H)$, then it is called  a \emph{Frobenius functor}. Note that $H$ is also a Frobenius functor, i.e. Frobenius functors always appear in pairs.

We refer to \cite[Definition 1.3.1]{Hov99} or Definition \ref{def:QuiFun} for the notion of \emph{Quillen adjunction}.

\begin{lemma}\label{lem:Qu-Adj}
Let $F: \mathcal{A}\leftrightarrows \mathcal{B} :H$ be an adjoint pair. Assume that $F$ is an exact functor. The following are equivalent:
\begin{enumerate}
\item $(F, H)$ is a Quillen adjunction between the model categories $\mathcal{GP}^{<\infty}(\mathcal{A})$ and $\mathcal{GP}^{<\infty}(\mathcal{B})$;
\item $F(\mathcal{GP}(\mathcal{A}))\subseteq \mathcal{GP}(\mathcal{B})$ and $F(\mathcal{P}(\mathcal{A}))\subseteq \mathcal{P}(\mathcal{B})$;
\item $H$ is an exact functor, $H(\mathcal{GP}^{<\infty}(\mathcal{B}))\subseteq \mathcal{GP}^{<\infty}(\mathcal{A})$ and $H(\mathcal{P}^{<\infty}(\mathcal{B}))\subseteq \mathcal{P}^{<\infty}(\mathcal{A})$.
\end{enumerate}
\end{lemma}

\begin{proof}
An adjoint pair of functors $(F, H)$ between model categories is a Quillen adjunction if $F$ preserves cofibrations and trivial cofibrations, or equivalently, $H$ preserves fibrations and trivial fibrations;
see \cite[Lemma 1.3.4]{Hov99}. Since $F$ is exact, it is direct that $F$ preserves cofibrations if and only if $F(\mathcal{GP}(\mathcal{A}))\subseteq \mathcal{GP}(\mathcal{B})$, and $F$ preserves trivial cofibrations if and only if $F(\mathcal{P}(\mathcal{A}))\subseteq \mathcal{P}(\mathcal{B})$. By \cite[Lemma 2.1]{CR20}, the functor $H$ is exact if and only if $F(\mathcal{P}(\mathcal{A}))\subseteq \mathcal{P}(\mathcal{B})$. Then $H$ preserves fibrations and trivial fibrations if and only if $H(\mathcal{GP}^{<\infty}(\mathcal{B}))\subseteq \mathcal{GP}^{<\infty}(\mathcal{A})$ and $H(\mathcal{P}^{<\infty}(\mathcal{B}))\subseteq \mathcal{P}^{<\infty}(\mathcal{A})$. This completes the proof.
\end{proof}

\begin{corollary}\label{cor:Qu-Adj}
Let $F: \mathcal{A}\leftrightarrows \mathcal{B} :H$ be an adjoint pair of exact functors. If $H(\mathcal{P}(\mathcal{B}))\subseteq \mathcal{P}(\mathcal{A})$, then  $(F, H)$ is a Quillen adjunction between the model categories $\mathcal{GP}^{<\infty}(\mathcal{A})$ and $\mathcal{GP}^{<\infty}(\mathcal{B})$.
\end{corollary}

\begin{proof}
The exactness of the functor $H$ forces that $F$ preserves the projective objects. It follows from \cite[Lemma 3.6]{CR20} that $F(\mathcal{GP}(\mathcal{A}))\subseteq \mathcal{GP}(\mathcal{B})$. Hence, the condition (2) of the above lemma satisfies.
\end{proof}

In the following, we use $\eta\colon {\rm Id}_\mathcal{A}\rightarrow HF$ and $\varepsilon\colon FH\rightarrow {\rm Id}_\mathcal{B}$ to denote the unit and the counit of the adjoint pair $F: \mathcal{A}\rightleftarrows \mathcal{B} :H$, respectively.

We refer to Definition \ref{def:QRrep} for the notions of the \emph{cofibrant replacement functor} $Q$ and the \emph{fibrant replacement functor} $R$. Specifically, as far as the above Gorenstein projective model structure is concerned, for any object $X\in \mathcal{GP}^{<\infty}(\mathcal{A})$, the {\em cofibrant replacement} $Q(X)$ coincides with the {\em special Gorenstein projective precover} of $X$. The existence of $Q(X)$ is guaranteed by the complete cotorsion pair $(\mathcal{GP}(\mathcal{A}), \mathcal{P}^{<\infty}(\mathcal{A}))$, and is also proved in \cite[Theorem 2.10]{Hol04}. The {\em fibrant replacement} $R(X)$ is precisely $X$ since every object in
$\mathcal{GP}^{<\infty}(\mathcal{A})$ is fibrant itself.

For a Quillen adjunction $(F, H)$, the \emph{total left derived functor} $LF$, and the \emph{total right derived functor} $RH$, are defined by $LF = \mathrm{Ho}(F)\circ \mathrm{Ho}(Q)$, and $RH = \mathrm{Ho}(H)\circ \mathrm{Ho}(R)$, respectively; see \cite[Definition 1.3.6]{Hov99} or Definition \ref{def:derFun}.

\begin{lemma}\label{lem:Qu-Equ}
Let $F: \mathcal{A}\leftrightarrows \mathcal{B} :H$ be an adjoint pair consisting of exact and faithful functors. Assume that $(F, H): \mathcal{GP}^{<\infty}(\mathcal{A})\leftrightarrows \mathcal{GP}^{<\infty}(\mathcal{B})$ is a Quillen adjunction. The following are equivalent:
\begin{enumerate}
\item $(F, H)$ is a Quillen equivalence between the model categories $\mathcal{GP}^{<\infty}(\mathcal{A})$ and $\mathcal{GP}^{<\infty}(\mathcal{B})$;
\item For any $X\in \mathcal{GP}(\mathcal{A})$ and any $Y\in \mathcal{GP}^{<\infty}(\mathcal{B})$, both $\mathrm{Coker}(\eta_X)$ and $\mathrm{Ker}(\varepsilon_Y)$ have finite projective dimension;
\item There is an adjoint equivalence of homotopy categories
$$\xymatrix@C=40pt{\mathrm{Ho}(\mathcal{GP}^{<\infty}(\mathcal{A}))\ar@<0.5ex>[r]^{LF} &\mathrm{Ho}(\mathcal{GP}^{<\infty}(\mathcal{B}))\ar@<0.5ex>[l]^{RH}, }$$
where ``$LF$`` is defined on objects by first taking special Gorenstein projective precover and then applying the functor $F$, and ``$RH$'' is defined on objects by applying the functor $H$.
\end{enumerate}
\end{lemma}

\begin{proof}
Since the equivalence $(1)\Longleftrightarrow (3)$ follows directly from \cite[Proposition 1.3.13]{Hov99}, it is sufficient to prove $(1)\Longleftrightarrow (2)$.

By \cite[Proposition 1.3.13]{Hov99}, $(F, H)$ is a Quillen equivalence if and only if for any cofibrant object $X\in \mathcal{GP}(\mathcal{A})$, and any fibrant object $Y\in \mathcal{GP}^{<\infty}(\mathcal{B})$, both the compositions
$$\xymatrix{X\ar[r]^{\eta_X\quad} &HF(X)\ar[r]^{} & HRF(X),} \text{~~and~~} \xymatrix{FQH(Y)\ar[r]^{} & FH(Y)\ar[r]^{\quad\varepsilon_Y} & Y}$$
are weak equivalences.

For the cofibrant object $X\in \mathcal{GP}(\mathcal{A})$, $F(X)\in \mathcal{GP}(\mathcal{B})\subseteq \mathcal{GP}^{<\infty}(\mathcal{B})$ is fibrant, then $RF(X) = F(X)$, and the identity $HF(X)\rightarrow HRF(X)$ is clearly a weak equivalence. For $H(Y)$, there exists an exact sequence
$$\xymatrix{0\ar[r] &K \ar[r] & M \ar[r]^{q\quad} & H(Y) \ar[r] & 0,}$$
where $M\in \mathcal{GP}(\mathcal{A})$ and $K\in \mathcal{P}^{<\infty}(\mathcal{A})$. Since $M$ is a cofibrant object and $q: M\rightarrow H(Y)$ is a trivial fibration, we take $QH(Y) = M$. By applying the exact functor $F$ to the above sequence, we get an epimorphism $F(q)$ with $\mathrm{Ker}(F(q)) = F(K)\in \mathcal{P}^{<\infty}(\mathcal{B})$, that is, $F(q)$ is also a trivial fibration. Hence, $F(q): FQH(Y)= F(M)\longrightarrow FH(Y)$ is a weak equivalence.

It remains to prove that the unit $\eta_X: X\longrightarrow HF(X)$ and the counit $\varepsilon_Y: FH(Y)\longrightarrow Y$ are weak equivalences. We observe that $H(\varepsilon_Y)$ is epic by the identity
$H(\varepsilon_Y)\circ \eta_{H(Y)}={\rm Id}_{H(Y)}$. Since $H$ is exact and faithful, it detects epimorphisms. Hence, we infer that $\varepsilon_Y$ is epic, that is, $\varepsilon_Y$ is a fibration. Then, $\varepsilon_Y$ is a weak equivalence if and only if it is a trivial fibration, or equivalently, $\mathrm{Ker}(\varepsilon_Y)\in \mathcal{P}^{<\infty}(\mathcal{B})$.

Since the functor $F$ is exact and faithful, the unit $\eta_X: X\rightarrow HF(X)$ is a monomorphism. For $\mathrm{Coker}(\eta_X)$, there is an exact sequence $0\rightarrow K'\rightarrow P\rightarrow \mathrm{Coker}(\eta_X)\rightarrow 0$, where $P$ is projective. We consider the following pullback:
$$\xymatrix{
& & 0\ar[d] & 0\ar[d]\\
& & K'\ar@{=}[r] \ar[d] & K'\ar[d]\\
0\ar[r] &X \ar@{=}[d] \ar[r]^{\alpha} &M \ar@{-->}[d]^{\beta} \ar@{-->}[r] & P\ar[r]\ar[d] &0\\
0\ar[r] &X \ar[r]^{\eta_X\quad} &HF(X) \ar[d] \ar[r] & \mathrm{Coker}(\eta_X)\ar[r]\ar[d] &0\\
& & 0 & 0
}$$
From the middle row we infer that $\alpha:X\rightarrow M$ is a trivial cofibration, and then it is a weak equivalence. The middle column yields that $\beta: M\rightarrow HF(X)$ is fibration. Hence, $\eta_{X} = \beta\alpha$ is a weak equivalence if and only if so is $\beta$. This is equivalent to that $\beta$ is a trivial fibration, that is, $K' = \mathrm{Ker}(\beta)\in \mathcal{P}^{<\infty}(\mathcal{A})$. From the most right column we infer that
$K' \in \mathcal{P}^{<\infty}(\mathcal{A})$ if and only if $\mathrm{Coker}(\eta_X)\in \mathcal{P}^{<\infty}(\mathcal{A})$. Hence, $\eta_{X}$ is a weak equivalence if and only if $\mathrm{Coker}(\eta_X)\in \mathcal{P}^{<\infty}(\mathcal{A})$. This completes the proof.
\end{proof}

We are now in a position to state the first main result of this paper; compare to \cite[Proposition 4.2 (1)]{CR20}.

\begin{theorem}\label{thm:FQE}
Let $F: \mathcal{A}\leftrightarrows \mathcal{B} :H$ be a Frobenius pair with both $F$ and $H$ faithful. The following are equivalent:
\begin{enumerate}
\item $(F, H)$ is a Quillen equivalence between the model categories $\mathcal{GP}^{<\infty}(\mathcal{A})$ and $\mathcal{GP}^{<\infty}(\mathcal{B})$;
\item For any $X\in \mathcal{GP}(\mathcal{A})$ and any $Y\in \mathcal{GP}^{<\infty}(\mathcal{B})$, both $\mathrm{Coker}(\eta_X)$ and $\mathrm{Ker}(\varepsilon_Y)$ have finite projective dimension;
\item There is an adjoint pair of triangle equivalent functors between the stable categories $$F: \underline{\mathcal{GP}}(\mathcal{A})\leftrightarrows \underline{\mathcal{GP}}(\mathcal{B}) :H.$$
\end{enumerate}
\end{theorem}

\begin{proof}
By \cite[Corollary 2.2]{CR20}, both $F$ and $H$ are exact functors satisfying $F(\mathcal{P}(\mathcal{A}))\subseteq \mathcal{P}(\mathcal{B})$ and $H(\mathcal{P}(\mathcal{B}))\subseteq \mathcal{P}(\mathcal{A})$. It follows from \cite[Theorem 3.2]{CR20} that  $F(\mathcal{GP}(\mathcal{A}))\subseteq \mathcal{GP}(\mathcal{B})$ and $H(\mathcal{GP}(\mathcal{B}))\subseteq \mathcal{GP}(\mathcal{A})$. Then, by Lemma \ref{lem:Qu-Adj},  $(F, H)$ is a Quillen adjunction between the model categories $\mathcal{GP}^{<\infty}(\mathcal{A})$ and $\mathcal{GP}^{<\infty}(\mathcal{B})$. Hence, $(1) \Longleftrightarrow (2)$ follows immediately from Lemma \ref{lem:Qu-Equ}.

By Proposition \ref{prop:GPModel} and Lemma \ref{lem:Qu-Equ}, $(F, H)$ is a Quillen equivalence if and only if there are equivalences of categories
$$\xymatrix@C=40pt{ \underline{\mathcal{GP}}(\mathcal{A})\ar@<0.5ex>[r]^{\mathrm{Ho}(i)\quad} &\mathrm{Ho}(\mathcal{GP}^{<\infty}(\mathcal{A}))\ar@<0.5ex>[r]^{LF} \ar@<0.5ex>[l]^{\mathrm{Ho}(Q)\quad}
&\mathrm{Ho}(\mathcal{GP}^{<\infty}(\mathcal{B}))\ar@<0.5ex>[l]^{RH} \ar@<-0.5ex>[r]_{\quad\mathrm{Ho}(Q)} &\underline{\mathcal{GP}}(\mathcal{B})\ar@<-0.5ex>[l]_{\quad\mathrm{Ho}(i)} , }$$
where $\mathrm{Ho}(i)$ is induced by the inclusion functor, and the inverse of $\mathrm{Ho}(i)$ is the composition of the functors $\mathrm{Ho}(R)\circ\mathrm{Ho}(Q) = \mathrm{Ho}(Q)$.
See also \cite[Chapter 1]{Hov99} for details in general cases. This yields the adjunction of equivalences $F: \underline{\mathcal{GP}}(\mathcal{A})\leftrightarrows \underline{\mathcal{GP}}(\mathcal{B}) :H.$

It suffices to show that $F$ is a triangle functor. Recall that for any $M, N\in \mathcal{GP}(\mathcal{A})$ and any morphism $M\rightarrow N$, there is an exact sequence $0\rightarrow M\rightarrow P\rightarrow M^{'}\rightarrow 0$, where $P\in \mathcal{P}(\mathcal{A})$ is the projective-injective object. Then we get a standard triangle $M\rightarrow N\rightarrow L\rightarrow M^{'}$ in the stable category $\underline{\mathcal{GP}}(\mathcal{A})$, by the following pushout of $M\rightarrow N$ and $M\rightarrow P$:
$$\xymatrix{
M\ar[r]\ar[d] &P \ar[d] \ar[r] &M^{'} \ar@{=}[d]\\
N\ar[r] &L \ar[r] &M^{'}
}$$
Since the Frobenius functor $F$ is exact such that $F(\mathcal{GP}(\mathcal{A}))\subseteq \mathcal{GP}(\mathcal{B})$ and $F(\mathcal{P}(\mathcal{A}))\subseteq \mathcal{P}(\mathcal{B})$, then $F$ preserves standard triangles, as expected. This completes the proof.
\end{proof}

For abelian category $\mathcal{A}$, denote by $\mathbf{D}^b(\mathcal{A})$ the bounded derived category. The bounded homotopy category $\mathbf{K}^b(\mathcal{P}(\mathcal{A}))$ can be viewed as a triangulated subcategory of  $\mathbf{D}^b(\mathcal{A})$. The \emph{singularity category} of $\mathcal{A}$ is defined to be the Verdier quotient triangulated category $\mathbf{D}_{\rm sg}(\mathcal{A})= \mathbf{D}^b(\mathcal{A})/{\mathbf{K}^b(\mathcal{P}(\mathcal{A}))}$. Note that $\mathbf{D}_{\rm sg}(\mathcal{A})$ vanishes if and only if each object in $\mathcal{A}$ has finite projective dimension.

\begin{corollary}\label{cor:SinEqu}
Let $F: \mathcal{A}\leftrightarrows \mathcal{B} :H$ be a Frobenius pair with both $F$ and $H$ faithful. If either $\mathcal{A}$ or $\mathcal{B}$ is a Gorenstein category, then the following are equivalent:\\
\begin{enumerate}
\item $(F, H)$ is a Quillen equivalence between the Gorenstein projective model categories $\mathcal{A}$ and $\mathcal{B}$;
\item There is an adjoint pair of triangle equivalent functors between the singular categories $$F: \mathbf{D}_{\rm sg}(\mathcal{A})\leftrightarrows \mathbf{D}_{\rm sg}(\mathcal{B}) :H.$$
\end{enumerate}
\end{corollary}

\begin{proof}
It follows from \cite[Corollary 3.3]{CR20} that  $\mathcal{A}$ is a Gorenstein category if and only if so is $\mathcal{B}$. Then the Gorenstein projective model structures of Proposition \ref{prop:GPModel} exist on both $\mathcal{A}$ and $\mathcal{B}$.

There is a canonical functor ${\rm can}\colon \underline{\mathcal{GP}}(\mathcal{A})\longrightarrow \mathbf{D}_{\rm sg}(\mathcal{A})$ sending each Gorenstein projective object to the corresponding stalk complex concentrated in degree zero. The functor is triangle; see for example \cite[Lemma 2.5]{Chen11}. Moreover, since the underlying category is Gorenstein, we infer from \cite[Theorem 6.9]{Bel00} or \cite[Theorem 3.3]{Chen11} that the canonical functor is an equivalence. Hence, the assertion follows from the above theorem.
\end{proof}

\begin{remark}
\indent$(1)$ For categories of modules, the Frobenius pair was introduced by Morita \cite{Mor65} with the name of strongly adjoint pair.

\indent$(2)$ The faithful Frobenius functors arise naturally in stable equivalences of Morita type; see for example \cite{DM07, LX07, Xi08}. As an application, this will be discussed further in Section 5.

\indent$(3)$ Recall that an extension of rings $R\rightarrow S$ is a \emph{Frobenius extension} \cite{Kas61} provided that $(S\otimes_R-, {\rm Res})$ is a Frobenius pair, where $S\otimes_R-: R\mbox{-}{\rm Mod}\rightarrow S\mbox{-}{\rm Mod}$ is the scalar-extension functor and ${\rm Res}: S\mbox{-}{\rm Mod}\rightarrow R\mbox{-}{\rm Mod}$ is the restriction functor.  Frobenius extensions are of broad interest in many areas such as topological quantum field theories in dimension 2, Calabi-Yau properties of Cherednik algebras, and flat-dominant dimensions; see for example \cite{BGS08, Kad99, Koc04, Xi21}.

Let $G$ be a group and $H\subseteq G$ be a subgroup of finite index. For group rings $RG$ and $RH$ over a ring $R$, the induction functor ${\rm Ind}_{H}^{G} = RG\otimes_{RH}-$ and the coinduction functor
${\rm Coind}_{H}^{G} = {\rm Hom}_{RH}(RG, -)$ are isomorphic, and then we have a classical Frobenius pair $({\rm Ind}_{H}^{G}, {\rm Res})$. The truncated polynomial extension $R\rightarrow R[x]/(x^t)$ ($t\geq 2$) is a Frobenius extension; the Gorenstein projective graded $R[x]/(x^2)$-modules are studied in \cite[Theorem 3.2]{Ren18}. Markov extensions are Frobenius extensions; see for example \cite[Chapter 3]{Kad99}.
\end{remark}

\section{\bf $K$-theory of Gorenstein projective modules}

In this section, we intend to apply Waldhausen's construction to define $K$-theory space and $K$-groups with respect to Gorenstein projective objects. These relative $K$-groups are invariant under Quillen equivalences.

\begin{lemma}\label{lem:WEqu}
Let $\mathcal{A}$ be an abelian category with enough projective objects, $X$ and $Y$ be any Gorenstein projective objects in $\mathcal{A}$. The following are equivalent:
\begin{enumerate}
\item There is a weak equivalence between $X$ and $Y$;
\item $X$ and $Y$ are isomorphic in the stable category $\underline{\mathcal{GP}}(\mathcal{A})$;
\item There is an isomorphism $X\oplus P\cong Y\oplus Q$ in $\mathcal{A}$ for some projective objects $P$ and $Q$.
\end{enumerate}
\end{lemma}

\begin{proof}
$(1)\Longrightarrow (2)$ It follows from Proposition \ref{prop:GPModel} that the homotopy category $\mathrm{Ho}(\mathcal{GP}^{<\infty}(\mathcal{A}))$ is equivalent to the stable category $\underline{\mathcal{GP}}(\mathcal{A})$, then the implication is clear.

$(2)\Longrightarrow (3)$ It follows by a standard argument in stable categories; see for example \cite[Lemma 1.1]{CZ07}. We include a proof for completeness. Let $f: X\rightarrow Y$ be an isomorphism in $\underline{\mathcal{GP}}(\mathcal{A})$. Then there exists a morphism $g: Y\rightarrow X$ such that $gf - {\rm Id}_X$ factors through some projective object $Q$. Namely, we have morphisms $\alpha: X\rightarrow Q$ and $\beta: Q\rightarrow X$, such that $gf - {\rm Id}_X = \beta\alpha$. Hence, $(g, -\beta)\left(\begin{matrix}f \\ \alpha \end{matrix}\right) = {\rm Id}_X$, and then there is an object $P$ and an isomorphism $h: X\oplus P\rightarrow Y\oplus Q$ in $\mathcal{A}$, such that $\left(\begin{matrix}f \\ \alpha \end{matrix}\right) = h\left(\begin{matrix}{\rm Id}_X \\ 0 \end{matrix}\right)$. In the stable category $\underline{\mathcal{GP}}(\mathcal{A})$, we infer from the isomorphism $f: X\rightarrow Y$ that $\left(\begin{matrix} {\rm Id}_X \\ 0 \end{matrix}\right): X\rightarrow X\oplus P$ is also an isomorphism. Hence, we have ${\rm Id}_P = 0$ in $\underline{\mathcal{GP}}(\mathcal{A})$, that is, $P$ is projective.

$(3)\Longrightarrow (1)$ Note that the isomorphism $X\oplus P\rightarrow  Y\oplus Q$ in $\mathcal{A}$ is a weak equivalence. Moreover, the injection $X\rightarrow X\oplus P$ is a trivial cofibration, and the projection $Y\oplus Q\rightarrow Y$ is a trivial fibration. By the concatenation of these maps, we get a weak equivalence between $X$ and $Y$. This completes the proof.
\end{proof}

In \cite{Wal85}, Waldhausen defined a notion of \emph{category with cofibrations and weak equivalences} and constructed a $K$-theory space from such category. We refer to \cite[Section 1.1 and 1.2]{Wal85} or
\cite[II, Definition 9.1.1]{Web13} for the notion of a \emph{Waldhausen category}.

Considering the Gorenstein projective model structure, we have the following.

\begin{lemma}\label{lem:Wad-GP}
For any pointed abelian category $\mathcal{A}$ with enough projective objects, the category of Gorenstein projective objects $\mathcal{GP}(\mathcal{A})$ is a Waldhausen category. Moreover, it satisfies the saturation axiom
(\cite[Section 1.2]{Wal85}), that is, if $f$, $g$ are composable maps in $\mathcal{GP}(\mathcal{A})$ and if two of $f$, $g$, $gf$ are weak equivalences, then so is the third.
\end{lemma}

\begin{proof}
We use $\rightarrowtail$ to denote the cofibration. It is clear that every isomorphism is both a cofibration and a weak equivalence. For any $X\in \mathcal{GP}(\mathcal{A})$, $0\rightarrowtail X$ is a cofibration. For any cofibration $X\stackrel{f}\rightarrowtail Y$ and any map $X\stackrel{g}\rightarrow Z$, the following commutative diagram
$$\xymatrix{
X\ar@{>->}[r]^{f}\ar[d]_{g} &Y \ar[d]_{t} \ar[r] &\mathrm{Coker}(f) \ar@{=}[d]\\
Z\ar@{>->}[r]^{ s \quad} &Y\bigsqcup_{X}Z \ar[r] &\mathrm{Coker}(f)
}$$
yields that the pushout $Y\bigsqcup_{X}Z$ of these two maps exists in $\mathcal{GP}(\mathcal{A})$, and moreover, the map $Z\rightarrowtail Y\bigsqcup_{X}Z$ is a cofibration. Hence, $\mathcal{GP}(\mathcal{A})$ is a category with cofibrations.

It is clear that weak equivalences are closed under composition. It remains to show that weak equivalences satisfy the ``gluing axiom'', that is, for every commutative diagram
$$\xymatrix@C=30pt@R=20pt{
Y \ar[d]_{\beta} & X\ar@{>->}[l]_{f} \ar[r]^{g} \ar[d]^{\alpha} &Z\ar[d]^{\gamma}\\
Y' &X' \ar@{>->}[l]_{f'} \ar[r]^{g'} &Z'
}$$
in which the vertical maps are weak equivalences and the two left horizontal maps are cofibrations, the induced map from the pushout $Y\bigsqcup_{X} Z$ to the pushout $Y'\bigsqcup_{X'} Z'$ is also a weak equivalence.
We denote by $s': Z'\rightarrowtail Y'\bigsqcup_{X'}Z'$ and $t': Y'\rightarrowtail Y'\bigsqcup_{X'}Z'$. Since
$$s'\gamma g = s'g'\alpha = t'f'\alpha = t'\beta f,$$
by the universal property of the pushout, there exists a unique map $\delta: Y\bigsqcup_{X} Z\rightarrow Y'\bigsqcup_{X'} Z'$ such that $s'\gamma = \delta s$ and $t'\beta = \delta t$.
Moreover, we have the following commutative diagram in $\mathcal{A}$
$$\xymatrix{
X \ar[r]^{\left(\begin{smallmatrix} f\\ -g \end{smallmatrix}\right)} \ar[d]_{\alpha}  &Y\oplus Z \ar[r]^{(t,s)}\ar[d]^{\left(\begin{smallmatrix} \beta &0\\0 &\gamma \end{smallmatrix}\right)} &Y\bigsqcup_{X} Z \ar[d]^{\delta}\\
X'\ar[r]_{\left(\begin{smallmatrix} f'\\ -g' \end{smallmatrix}\right)} &Y'\oplus Z' \ar[r]_{(t', s')} &Y'\bigsqcup_{X'} Z'
}$$
which can also be considered as a commutative diagram in the stable category $\underline{\mathcal{GP}}(\mathcal{A})$. It is easy to see that for weak equivalences $\beta$ and $\gamma$, the induced map $\left(\begin{smallmatrix} \beta &0\\0 &\gamma \end{smallmatrix}\right)$ is also a weak equivalence. By Lemma \ref{lem:WEqu}, the vertical maps $\alpha$ and $\left(\begin{smallmatrix} \beta &0\\0 &\gamma \end{smallmatrix}\right)$ are isomorphism in $\underline{\mathcal{GP}}(\mathcal{A})$. Hence, $\delta: Y\bigsqcup_{X} Z\rightarrow Y'\bigsqcup_{X'} Z'$ is also an isomorphism in $\underline{\mathcal{GP}}(\mathcal{A})$. Again, we infer from Lemma \ref{lem:WEqu} that $\delta$ is a weak equivalence. This yields that $\mathcal{GP}(\mathcal{A})$ is a category with weak equivalences.

It follows immediately from the definition of model categories that weak equivalences satisfy 2-out-of-3 property; see \cite[Definition 1.1.3]{Hov99} or Definition \ref{def:MCat}. Hence, the saturation axiom holds.
\end{proof}

For a Waldhausen category $\mathcal{C}$, we denote by $w\mathcal{C}$ the family of weak equivalences in $\mathcal{C}$. Since every isomorphism is a weak equivalence and weak equivalences are closed under composition, we may regard  $w\mathcal{C}$ as a subcategory of $\mathcal{C}$.

For convenience and completeness, we will briefly recall Waldhausen's construction and the geometric realization of simplicial sets later in Appendix A, and refer to \cite{Wal85} and \cite{Web13} for details. For the category $w\mathcal{C}$, one may apply Waldhausen's $S_{\bullet}$-construction to obtain a simplicial category $wS_{\bullet}\mathcal{C}$. Denote by $|wS_{\bullet}\mathcal{C}|$ the realization of $wS_{\bullet}\mathcal{C}$.
In particular, let $\mathcal{C} = \mathcal{GP}(\mathcal{A})$, and then we have the following by \cite[IV, Definition 8.5]{Web13}; see also Definition \ref{def:K-group}.

\begin{definition}\label{def:GK-group}
For the Waldhausen category $\mathcal{C} =\mathcal{GP}(\mathcal{A})$, its $K$-theory space is the loop space $\Omega |wS_{\bullet}\mathcal{C}|$. For $i\geq 0$, the $K$-groups of $\mathcal{GP}(\mathcal{A})$ are defined to be the homotopy groups:
$$K_{i}(\mathcal{GP}(\mathcal{A})) = \pi_i(\Omega |wS_{\bullet}\mathcal{C}|) = \pi_{i+1}(|wS_{\bullet}\mathcal{C}|).$$
\end{definition}

In the following, we assume that $k$ is a fixed field, $A$ and $B$ are finite dimensional algebras with unit over $k$. Let $\mathcal{A}= A\mbox{-mod}$ and $\mathcal{B}= B\mbox{-mod}$ be the categories of finitely generated left modules over $A$ and $B$, respectively. We abuse the notations $\mathcal{GP}(A)$ and $\mathcal{GP}(B)$ to denote the subcategories of finitely generated Gorenstein projective modules over $A$ and $B$, respectively. The $K$-group
$K_i(\mathcal{GP}(A))$ is called $i^{th}$ \emph{Gorenstein $K$-group} of $A$, and is denoted by $K_i^G(A)$ for simplicity, where the superscript ``$G$'' stands for ``Gorenstein''.

To simplify the notational burden, we introduce the following.

\begin{definition}\label{def:G-staFroF}
Let $F: \mathcal{A}\leftrightarrows \mathcal{B} :H$ be a Frobenius pair of faithful functors between categories of finitely generated left modules over finite dimensional algebras $A$ and $B$. If it satisfies one of the equivalence conditions in Theorem \ref{thm:FQE}, then $(F, H)$ is called a Gorenstein stable Frobenius pair between $A$ and $B$, and $F$ (or $G$) is called a Gorenstein stable Frobenius functor.
\end{definition}

We will show in the next section that a stable equivalence of Morita type between $A$ and $B$ naturally induces a Gorenstein stable Frobenius pair between $A$ and $B$; see Proposition \ref{prop:AdjType}.

The following is immediate from Theorem \ref{thm:FQE} together with \cite[Corollary 3.10]{DS04}.

\begin{proposition}\label{prop:StaK-group}
Let $(F, H)$ be a Gorenstein stable Frobenius pair between $A$ and $B$. Then for any $i\geq 0$, there is an isomorphism of Gorenstein $K$-groups $K_i^G(A)\cong K_i^G(B)$.
\end{proposition}

It is well known that Morita equivalences preserve classical $K$-groups, while it is unknown whether stable equivalences of Morita type preserve classical $K$-groups. By specifying the above result, we will see in the next section that Gorenstein $K$-groups are invariant under stable equivalences of Morita type.

By Waldhausen's construction, $S_0\mathcal{GP}(A)$ is trivial, then $|wS_{\bullet}\mathcal{GP}(A)|$ is a connected space. The object of the category $S_1\mathcal{GP}(A)$ is of the form $0\rightarrowtail G$ with $G\in \mathcal{GP}(A)$, and is thus isomorphic to $\mathcal{GP}(A)$. Hence $\pi_0(|wS_1\mathcal{GP}(A)|)$ is the set of weak equivalence classes of objects in $\mathcal{GP}(A)$.  The object of the category $S_2\mathcal{GP}(A)$ is of the form $0\rightarrowtail G_1 \rightarrowtail G_2$ together with $G_2/G_1$, and then $\pi_0(|wS_2\mathcal{GP}(A)|)$ is the set of weak equivalence classes of cofibration sequences. There are face maps
$d_i: S_2\mathcal{GP}(A)\rightarrow S_1\mathcal{GP}(A)$ ($0\leq i\leq 2$), such that for any object $\alpha = G_1 \rightarrowtail G_2\twoheadrightarrow G_2/G_1$ of $S_2\mathcal{GP}(A)$, one has
$d_0(\alpha) = G_2/G_1$, $d_1(\alpha) = G_2$ and  $d_2(\alpha) = G_1$. Hence, we have the following; see also \cite[II Definition 9.1.2]{Web13}.

\begin{definition}\label{def:K0}
For a finite dimensional algebra $A$, $K_0^G(A)$ is the Grothendieck group of weak equivalence classes of finitely generated Gorenstein projective modules over $A$ (or equivalently, the isomorphism classes in the Gorenstein stable categories $\underline{\mathcal{GP}}(A)$). That is, $K_0^G(A)$ is generated by the set of weak equivalence classes $[X]$ of $X\in \mathcal{GP}(A)$ with the relation that $[Y] = [X] + [Z]$ for any cofibration sequence
$X\rightarrowtail Y\twoheadrightarrow Z$.
\end{definition}

Note that each weak self-equivalence $\alpha: X\rightarrow X$ of $X\in \mathcal{GP}(A)$ yields an element of $\pi_{2}(|wS_{\bullet}\mathcal{GP}(A)|)$. In order to distinguish the weak self-equivalences of different modules, we denote these elements by pairs $(X, \alpha)$.

\begin{definition}\label{def:K1}
Let $A$ be a finite dimensional algebra. $K_1^G(A)$ is the free abelian group on pairs $(X, \alpha)$, where $X\in \mathcal{GP}(A)$ and $\alpha$ is a weak self-equivalence of $X$, modulo the following relations:
 \begin{enumerate}
\item $[(X, \alpha)] + [(X, \beta)] =[(X,\alpha\beta)]$;
\item If there is a commutative diagram of cofibration sequences
$$\xymatrix{
X\ar@{>->}[r]^{f}\ar[d]_{\alpha} &Y \ar[d]_{\beta} \ar@{->>}[r]^{g} &Z \ar[d]_{\gamma}\\
X\ar@{>->}[r]^{f} &Y\ar@{->>}[r]^{g} &Z
}$$
where $\alpha$, $\beta$ and $\gamma$ are weak self-equivalences, then
$$[(X, \alpha)] + [(Z, \gamma)] = [(Y, \beta)].$$
\end{enumerate}
\end{definition}

An algebra $A$ is \emph{CM-finite} provided that there are only finitely many isomorphism classes of indecomposable finitely generated Gorenstein projective $A$-modules; $A$ is \emph{CM-free} if all Gorenstein projective modules in $A$-mod are projective. It is clear that algebras of finite global dimension are CM-free. If $A$ is an Artin algebra, then $A$ is CM-finite if and only if there exists a Gorenstein projective module $G\in A\mbox{-mod}$ such that $\mathcal{GP}(A) = \mathrm{add}G$. We refer to \cite{Rin13} for examples of CM-finite algebras and CM-free algebras of infinite global dimension.

Before characterizing $K_1^G(A)$ for CM-finite Artin algebra $A$, we first show that CM-free and CM-finite properties are invariant under Gorenstein stable Frobenius functors.

\begin{proposition}\label{prop:CM-inv}
Let $(F, H)$ be a Gorenstein stable Frobenius pair between $A$ and $B$.
\begin{enumerate}
\item $A$ is CM-free if and only if $B$ is CM-free.
\item If $A$ and $B$ are Artin algebras, then $A$ is CM-finite if and only if $B$ is CM-finite.
\end{enumerate}
\end{proposition}

\begin{proof}
Note that $A$ is CM-free if $\underline{\mathcal{GP}}(A) = 0$. The assertion for CM-freeness is immediate since there are mutually inverse equivalences
$F: \underline{\mathcal{GP}}(A)\leftrightarrows \underline{\mathcal{GP}}(B) :H$ by Theorem \ref{thm:FQE}.

Now we assume that $A$ is a CM-finite Artin algebra, and $\mathcal{GP}(A) = \mathrm{add}G$ for a given Gorenstein projective $A$-module $G$. We claim that $\mathcal{GP}(B) = \mathrm{add}F(G)$, which implies that $B$ is CM-finite. Since $H$ is also a Gorenstein stable Frobenius functor, by a similar argument, we can prove that if $B$ is CM-finite, then so is $A$.

Clearly, $F(G)$ is a Gorenstein projective $B$-module, and then $\mathrm{add}F(G) \subseteq\mathcal{GP}(B)$. Let $Y\in  \mathcal{GP}(B)$ be any finitely generated Gorenstein projective $B$-module. Then $H(Y)\in \mathcal{GP}(A) = \mathrm{add}G$, and moreover, $FH(Y)\in  \mathrm{add}F(G)$. By Theorem \ref{thm:FQE}, there is an exact sequence of $B$-modules
$$0\longrightarrow \mathrm{Ker}(\varepsilon_Y)\longrightarrow FH(Y)\stackrel{\varepsilon_Y}\longrightarrow Y\longrightarrow 0$$
with $\mathrm{Ker}(\varepsilon_Y)$ being of finite projective dimension, where $\varepsilon$ is the counit of the adjoint pair $(F, H)$. Since $\mathrm{Ker}(\varepsilon_Y)$ is also Gorenstein projective, we get that $\mathrm{Ker}(\varepsilon_Y)$ is projective. Since $Y\in  \mathcal{GP}(B)$, the above sequence is split, and then we infer that $Y\in  \mathrm{add}F(G)$ because $Y$ is a direct summand of $FH(Y)$. This yields $\mathcal{GP}(B)\subseteq \mathrm{add}F(G)$, and finally, $\mathcal{GP}(B)= \mathrm{add}F(G)$.
\end{proof}

Let $\Lambda$ be a ring with unit. We denote by $GL_{n}(\Lambda)$ the group of $n\times n$ matrices over $\Lambda$. There is a group homomorphism $M\rightarrow \left(\begin{smallmatrix} M &0\\ 0 &1 \end{smallmatrix}\right)$ from  $GL_{n}(\Lambda)$ to  $GL_{n+1}(\Lambda)$. Denote by $GL(\Lambda):= \underrightarrow{\mathrm{lim}}_{n}GL_{n}(\Lambda)$.

The $n\times n$ elementary matrix $e_{ij}(\lambda)$ is defined by 1's on the diagonal, $\lambda\in \Lambda$ in the $(i,j)$-spot, and 0's elsewhere. The subgroup of $GL_{n}(\Lambda)$ generated by elementary matrices is denoted by $E_{n}(\Lambda)$. Denote by $E(\Lambda):= \underrightarrow{\mathrm{lim}}_{n}E_{n}(\Lambda)$ the {\em group of elementary matrices}.

By Whitehead's Lemma, the commutator subgroups of $GL(\Lambda)$ and of $E(\Lambda)$ coincide with $E(\Lambda)$, and moreover, $E(\Lambda)$ is normal in $GL(\Lambda)$; see for example \cite[Proposition 2.1.4]{Ros94}. The classical $K_1$-group $K_1(\Lambda)$ of $\Lambda$ is defined to be the abelian group $GL(\Lambda)/E(\Lambda)$.

We conclude this section by a characterization of $K_1^G(A)$ for the CM-finite Artin algebra $A$.

\begin{theorem}\label{thm:CharK1}
Assume $A$ is a CM-finite Artin algebra, and $\mathcal{GP}(A) = \mathrm{add}G$ for some Gorenstein projective $A$-module $G$. Let $\Lambda = \underline{\mathrm{End}}(G)$ be the endomorphism ring of $G$ in the stable category $\underline{\mathcal{GP}}(A)$. There is an isomorphism $K_1^G(A)\cong K_1(\Lambda)$.
\end{theorem}

\begin{proof}
If $M\in GL_{m}(\Lambda)$, then $M$ defines an automorphism $\alpha$ of $G^m$ in $\underline{\mathcal{GP}}(A)$, or equivalently, a weak self-equivalence of $G^m$, which is denoted by $f_m(M) =[(G^m, \alpha)]$. Let $\mathrm{Id}_{n}$ be the identity map of $G^n$. Then there is a map $f_{m+n}\left(\begin{smallmatrix} M &0\\ 0 &1_{n} \end{smallmatrix}\right) = [(G^{m+n}, \alpha\oplus \mathrm{Id}_n)]$. For any $X\in \mathcal{GP}(A)$, we observe that $[(X, \mathrm{Id})] + [(X, \mathrm{Id})] =[(X, \mathrm{Id}\circ \mathrm{Id})] = [(X, \mathrm{Id})]$, and then $[(X, \mathrm{Id})] = 0$. Hence, $[(G^{m+n}, \alpha\oplus \mathrm{Id}_n)] = [(G^{m}, \alpha)] + [(G^{n}, \mathrm{Id}_n)] = [(G^{m}, \alpha)]$. Then, there is a commutative diagram
$$\xymatrix{
GL_{n}(\Lambda)\ar[rr]^{}\ar[rd]_{f_m} && GL_{m+n}(\Lambda) \ar[ld]^{f_{m+n}}\\
&K^G_1(A)
}$$
Hence, we have a map $f= \underrightarrow{\mathrm{lim}}f_{n}: GL(\Lambda)\rightarrow K^G_1(A)$.

For any elementary matrix $e_{ij}(\lambda)\in E_l(\Lambda)$, there is a commutative diagram
$$\xymatrix{
G^{l-1}\ar@{>->}[r]^{\rho}\ar@{=}[d] &G^{l} \ar[d]_{e_{ij}(\lambda)} \ar@{->>}[r]^{\pi} &G \ar@{=}[d]\\
G^{l-1}\ar@{>->}[r]^{\rho} &G^l \ar@{->>}[r]^{\pi} &G
}$$
where $\pi$ is a projection onto the $i$-th coordinate, and $\rho$ maps $(g_1,\cdots,\widehat{g_i},\cdots, g_l)$ to $(g_1,\cdots,0,\cdots, g_l)$. Then, by Definition \ref{def:K1} it follows that
$[(G^l, e_{ij}(\lambda))] = [(G^{l-1}, \mathrm{Id}_{l-1})] + [(G, \mathrm{Id}_{1})] = [(G^l, \mathrm{Id}_{l})]$. Moreover, for any $L\in E_{l}(\Lambda)$, we have $f(L)= [(G^l, \mathrm{Id}_{l})] = 0$.

Now, let $\varphi: GL(\Lambda)/E(\Lambda)\rightarrow K_1^G(A)$ be a map defined by $[M]\rightarrow f(M) = [(G^m, \alpha)]$. Let $N\in GL_{n}(\Lambda)$, which defines an automorphism $\beta: G^n\rightarrow G^n$ in $\underline{\mathcal{GP}}(A)$. Assume $[M] = [N]$, that is, there is some $l\geq \rm{max}\{m, n\}$, such that
$\left(\begin{smallmatrix} M &0\\ 0 &1_{l-m} \end{smallmatrix}\right) \left(\begin{smallmatrix} N &0\\ 0 &1_{l-n} \end{smallmatrix}\right)^{-1} \in E_l(\Lambda)$.  Then we have
$$[(G^l, \alpha\oplus \mathrm{Id}_{l-m})] - [(G^l, \beta\oplus \mathrm{Id}_{l-n})] = 0.$$
Note that $[(G^m, \alpha)] = [(G^l, \alpha\oplus \mathrm{Id}_{l-m})]$ and $[(G^n, \beta)] = [(G^l, \beta\oplus \mathrm{Id}_{l-n})]$. Hence, $\varphi([M]) = \varphi([N])$. This implies that the map
$\varphi: GL(\Lambda)/E(\Lambda)\rightarrow K_1^G(A)$ is well-defined.

It is easy to see that $\varphi$ is a group homomorphism. If there is a matrix $M'\in GL_{m}(\Lambda)$ which defines a weak self-equivalence $\alpha'$ of $G^m$, then the equality $[(G^m, \alpha)] + [(G^m, \alpha')] = [(G^m, \alpha'\alpha)]$ holds immediately by Definition \ref{def:K1}.

Let $[(X, \alpha)]$ be any element of $K_1^G(A)$, where $\alpha$ is a weak self-equivalence of $X\in \mathcal{GP}(A)$. Since $\mathcal{GP}(A) = \mathrm{add}G$, there exists $Y\in \mathcal{GP}(A)$ such that
$X\oplus Y\cong G^n$. Then $[(X, \alpha)] = [(X, \alpha)] + [(Y, \mathrm{Id}_Y)] = [(G^n, \alpha\oplus \mathrm{Id}_Y)]$. Assume $X\oplus Z\cong G^m$, then $G^n\oplus Z \cong X\oplus Y\oplus Z \cong G^m \oplus Y$, and moreover,
$[(X, \alpha)] = [(G^{n+m}, \alpha\oplus \mathrm{Id}_{G^n\oplus Z})] = [(G^{n+m}, \alpha\oplus \mathrm{Id}_{G^m\oplus Y})]$. Hence, there exists a map $g_X: W(X)\rightarrow GL(\Lambda)/E(\Lambda)$ which only depends on $(X, \alpha)$, where $W(X)$ is the set of weak self-equivalences of $X$. We abuse the notation to denote $g_X(\alpha)$ by the matrix $\left(\begin{smallmatrix} \alpha &0\\ 0 &\mathrm{Id} \end{smallmatrix}\right)$.

If $[(X, \alpha)] + [(X, \beta)] =[(X, \alpha\beta)]$, then it follows from
$$\left(\begin{matrix} \alpha\beta &0\\ 0 &\mathrm{Id} \end{matrix}\right) =  \left(\begin{matrix} \alpha &0\\ 0 &\mathrm{Id} \end{matrix}\right) \left(\begin{matrix} \beta &0\\ 0 &\mathrm{Id} \end{matrix}\right)$$
that $g_X(\alpha\beta) = g_X(\alpha) g_X(\beta)$.

Consider the commutative diagram of cofibration sequences in Definition \ref{def:K1}:
$$\xymatrix{
X\ar@{>->}[r]^{f}\ar[d]_{\alpha} &Y \ar[d]_{\beta} \ar@{->>}[r]^{g} &Z \ar[d]_{\gamma}\\
X\ar@{>->}[r]^{f} &Y\ar@{->>}[r]^{g} &Z
}$$
where $\alpha$, $\beta$ and $\gamma$ are weak self-equivalences. Let $\alpha' = g_X(\alpha) = \left(\begin{smallmatrix} \alpha &0\\ 0 &\mathrm{Id} \end{smallmatrix}\right)$,
$\beta' = g_Y(\beta)$ and $\gamma' = g_Z(\gamma)$. Then there is a commutative diagram of cofibration sequences
$$\xymatrix{
G^n\ar@{>->}[r]^{f'}\ar[d]_{\alpha'} &G^{n+m} \ar[d]_{\beta'} \ar@{->>}[r]^{g'} &G^{m} \ar[d]_{\gamma'}\\
G^{n}\ar@{>->}[r]^{f'} &G^{n+m}\ar@{->>}[r]^{g'} &G^m
}$$
where $f' = \left(\begin{smallmatrix} f &0\\ 0 &0 \end{smallmatrix}\right)$,  $g' = \left(\begin{smallmatrix} g &0\\ 0 &0 \end{smallmatrix}\right)$.  Since $G^{n+m} = G^n\oplus G^m$, the matrix $\beta'$ can be represented by
$$\left(\begin{matrix} \alpha' &\alpha'\ast\\ 0 &\gamma' \end{matrix}\right) =  \left(\begin{matrix} \alpha' &0\\ 0 &\gamma'\end{matrix}\right) \left(\begin{matrix} \mathrm{Id} & \ast\\ 0 &\mathrm{Id} \end{matrix}\right),$$
where $\left(\begin{smallmatrix} \mathrm{Id} & \ast\\ 0 &\mathrm{Id} \end{smallmatrix}\right)\in E(\Lambda)$. Moreover, we have
$$\left(\begin{matrix} \alpha' &0\\ 0 &\gamma' \end{matrix}\right) =  \left(\begin{matrix} \alpha' &0\\ 0 &\mathrm{Id}\end{matrix}\right) \left(\begin{matrix} 0 & -\mathrm{Id}\\ \mathrm{Id} &0 \end{matrix}\right)
\left(\begin{matrix} \gamma' &0\\ 0 &\mathrm{Id}\end{matrix}\right) \left(\begin{matrix} 0 &\mathrm{Id}\\ -\mathrm{Id} &0 \end{matrix}\right).$$
Then $g_Y(\beta) = g_X(\alpha)g_Z(\gamma)$. Hence, we can define a group homomorphism
$\psi: K_1^G(A) \rightarrow  GL(\Lambda)/E(\Lambda)$ by $\psi(\alpha) = g_{X}(\alpha)$.

It is direct to check that $\varphi$ and $\psi$ are mutually inverse. Then, we get the desired isomorphism $K_1^G(A)\cong GL(\Lambda)/E(\Lambda) =  K_1(\Lambda)$.
\end{proof}

\section{\bf Applications: stable equivalences of Morita type}

In this section, we will show that if there is a stable equivalence of Morita type between two finite dimensional algebras, then Theorem \ref{thm:FQE} holds. In this case, $K$-theory of Gorenstein projective modules, and some properties such as CM-freeness and CM-finiteness, are invariant.

Let $A$ and $B$ be rings. Recall that a $B$-$A$-bimodule $M$ is a \emph{Frobenuis bimodule}, if both $_{B}M$ and $M_{A}$ are finitely generated projective modules, and there is an $A$-$B$-bimodule isomorphism $${^{*}M}:= \mathrm{Hom}_{B}(M, B)\cong \mathrm{Hom}_{A^{op}}(M, A): = M^{*}.$$
Denote the common $A$-$B$-bimodule by $N$. Then $(M\otimes_{A}-, N\otimes_{B}-)$ is a Frobenius pair between $A$-Mod and $B$-Mod. Indeed, it follows from \cite[Theorem 2.1]{CIGTN99} that any classical Frobenius pair between two module categories are of this form, that is, for any Frobenius functor $F$ between categories of modules, one has $F \cong M\otimes -$ for a Frobenius bimodule $M$. In particular, $\theta: R\rightarrow S$ is a Frobenius extension of rings, if and only if $S$ is a Frobenius $R$-$S$-bimodule.

Recall that a ring is \emph{Gorenstein regular} if it is of finite global Gorenstein projective dimension; in this case, the category of modules is a Gorenstein category. The following is immediate from \cite[Theorem 3.2, Corollary 3.3]{CR20}, which implies that we can get some new Gorenstein categories from the given ones via faithful Frobenius functors.

\begin{lemma}\label{lem:GpMod}
Let $_{B}M_{A}$ be a Frobenuis bimodule.
\begin{enumerate}
\item Let $X$ be any left $A$-module. If $X\in \mathcal{GP}(A)$, then $M\otimes_{A} X \in\mathcal{GP}(B)$; the converse holds if $M_{A}$ is a generator.
\item If $_{B}M$ and $M_{A}$ are generators, then $A$ is Gorenstein regular if and only if $B$ is Gorenstein regular.
\end{enumerate}
\end{lemma}

We remark that Frobenius bimodules arise naturally in stable equivalences of Morita type; see for example \cite{DM07, Xi08}. In the rest of the paper, we denote by $k$ a fixed field, and all algebras are assumed to be finite dimensional $k$-algebras with unit. By a module we always mean a finitely generated left modules.

\begin{definition}\label{def: StaEMT} (\cite{Bro92})
Let $A$ and $B$ be algebras. We say that two bimodules $_{A}N_{B}$ and $_{B}M_{A}$ define a \emph{stable equivalence of Morita type} between $A$ and $B$ if both $M$ and $N$ are projective both as left modules and as right modules, and if $N\otimes_{B} M \cong A\oplus P$ as $A$-$A$-bimodules for some projective $A$-$A$-bimodule $P$, and $M\otimes_{A} N\cong B\oplus Q$ as $B$-$B$-bimodules for some projective $B$-$B$-bimodule $Q$.
\end{definition}

The following facts are standard.

\begin{lemma}\label{lem:facts} Keep the notations as above. Then the following statements hold.
\begin{enumerate}
\item The images of the functors $T_{P} = P\otimes_{A}-$ and $T_{Q} = Q\otimes_{B}-$ consist of projective modules.
\item The bimodules $_{B}M_{A}$ and $_{A}N_{B}$ are projective generators both as left modules and as right modules. Therefore, the functors $T_{M} = M\otimes_{A}-$ and $T_{N} = N\otimes_{B}-$ are faithful.
\item $T_{N}\circ T_{M}\rightarrow \mathrm{Id}_{A\text{-}\mathrm{mod}}\oplus T_{P}$ and $T_{M}\circ T_{N}\rightarrow \mathrm{Id}_{B\text{-}\mathrm{mod}}\oplus T_{Q}$ are natural isomorphisms.
\item If $(M\otimes_{A}-, N\otimes_{B}-)$ and $(N\otimes_{B}-, M\otimes_{A}-)$ are adjoint pairs, then both $M$ and $N$ are Frobenius bimodules.
\end{enumerate}
\end{lemma}

\begin{proof}
For (1), we refer to \cite[Corollary 3.3]{AR90}, and for (2), we refer to \cite[Lemma 2.2]{Liu03}. For (3), we refer to the argument in \cite[Theorem 4.1]{Xi02}, and (4) is immediate from \cite[Lemma 3.3]{Xi08}.
\end{proof}

By \cite[Definition 3.2]{Xi08}, if a stable equivalence of Morita type between two algebras $A$ and $B$ defined by two bimodules $_{B}M_{A}$ and $_{A}N_{B}$ satisfies that both $(M\otimes_{A}-, N\otimes_{B}-)$ and $(N\otimes_{B}-, M\otimes_{A}-)$ are adjoint pairs, then it is called a \emph{stable equivalence of adjoint type}. By \cite{Ric91}, any derived equivalence between self-injective algebras induces a stable equivalence of adjoint type. Actually, there are plenty examples of stable equivalences of adjoint type outside the scope of self-injective algebras; see for example \cite{DM07, LX07, Xi08}.

We have a specific example of Gorenstein stable Frobenius pair as follows.

\begin{proposition}\label{prop:AdjType}
Suppose that two bimodules  $_{B}M_{A}$ and $_{A}N_{B}$ define a stable equivalence of Morita type between algebras $A$ and $B$ with both $(M\otimes_{A}-, N\otimes_{B}-)$ and $(N\otimes_{B}-, M\otimes_{A}-)$ being adjoint pairs. Then the adjunction
$$\xymatrix{ T_{M}=M\otimes_{A}-: \mathcal{GP}^{<\infty}(A) \ar@<0.5ex>[r] &\mathcal{GP}^{<\infty}(B) : T_{N} = N\otimes_{B}- \ar@<0.5ex>[l] }$$
is a Quillen equivalence between the Gorenstein projective model categories over $A$ and $B$, and furthermore, it induces equivalences of triangulated categories
$$\underline{\mathcal{GP}}(A)\simeq \mathrm{Ho}(\mathcal{GP}^{<\infty}(A)) \simeq \mathrm{Ho}(\mathcal{GP}^{<\infty}(B))\simeq \underline{\mathcal{GP}}(B).$$
\end{proposition}

\begin{proof}
By Lemma \ref{lem:facts}, $T_{M}: A\text{-mod} \rightleftarrows B\text{-mod}: T_{N}$ is a Frobenius pair of exact and faithful functors. For any $A$-module $X$, $T_{N}\circ T_{M}(X)\cong X\oplus  T_{P}(X)$, and then  the cokernel of the unit $\eta_X: X\longrightarrow T_{N}\circ T_{M}(X)$ is isomorphic to the projective $A$-module $T_{P}(X)$. Similarly, for any $B$-module $Y$, the kernel of the counit $\varepsilon_{Y}: T_{M}\circ T_{N}(Y)\rightarrow Y$ is isomorphic to the projective $B$-module $T_{Q}(Y)$. Then, the assertion follows immediately from Theorem \ref{thm:FQE}.
\end{proof}

By \cite[Corollary 5.5]{Ric91}, if two self-injective algebras are derived equivalent, then there are stable equivalence of Morita type. A typical example of self-injective algebras is the block algebras of group algebras, which is related to Brou\'{e}'s defect group conjecture. For any self-injective algebra $A$, it is well known that $A\mbox{-mod} = \mathcal{GP}(A)$. In this case,  $A\mbox{-mod}$ is a model category, where the subcategories of cofibrant and fibrant objects are both equal to $A\mbox{-mod}$, and the subcategory of trivial objects consists exactly of all finitely generated projective $A$-modules.

\begin{corollary}\label{cor:selfInj}
Let $A$ and $B$ be two finite dimensional self-injective algebras. If $A$ and $B$ are derived equivalent, then there is a Gorenstein stable Frobenius pair between $A$ and $B$. Moreover, the mutually inverse equivalence between the stable categories $A\mbox{-}\underline{\mathrm{mod}}$ and $B\mbox{-}\underline{\mathrm{mod}}$ is induced from a Quillen equivalence.
\end{corollary}

\begin{corollary}\label{cor:SSFroExt}
If $B$ and $\Lambda$ are algebras such that their semisimple quotients are separable and if at least one of them is indecomposable. Then there is a $k$-algebra $A$, Morita equivalent to $\Lambda$, such that $B\rightarrow A$ is a split and separable Frobenius extension. Moreover, there is a stable equivalence of Morita type between $A$ and $B$ with $(Res, F)$ and $(F, Res)$ being adjoint pairs, where $F = A\otimes_{B}-$ is the induction functor.
In this case, $(Res, F)$ is a Quillen equivalence which induces the equivalence $\underline{\mathcal{GP}}(A)\simeq \underline{\mathcal{GP}}(B)$ of stable categories.
\end{corollary}

\begin{proof}
It follows from \cite[Corollary 5.1]{DM07} that there is a $k$-algebra $A$ Morita equivalent to $\Lambda$, and an injective ring homomorphism $B\rightarrow A$, such that
$A\otimes_{B} A \cong A\oplus P$ as $A$-$A$-bimodules for some projective $A$-$A$-bimodule $P$, and $A\cong B\oplus Q$ as $B$-$B$-bimodules for some projective $B$-$B$-bimodule $Q$. This means that $_{B}A_{A}$ and $_{A}A_{B}$ define a stable equivalence of Morita type between $A$ and $B$. Hence, $B\rightarrow A$ is a split and separable Frobenius extension; see for example \cite[pp. 96]{HX18}. By Proposition \ref{prop:AdjType}, we get the adjunction of equivalences
$\xymatrix{ Res: \underline{\mathcal{GP}}(A) \ar@<0.5ex>[r] &\underline{\mathcal{GP}}(B) : F \ar@<0.5ex>[l] }$.
\end{proof}

There are some invariants preserved under stable equivalences of Morita type. For example, it follows from \cite{Xi08} that if there is a stable equivalence of Morita type between algebras $A$ and $B$, then the Hochschild cohomology groups and the absolute values of Cartan determinants of algebras $A$ and $B$ are invariant; moreover, it follows from \cite[Theorem 1.3]{LX07} that the self-injective dimensions of $A$ and $B$ are equal.

We denote by $\mathrm{spli}(A)$ and $\mathrm{silp}(A)$ the supremum of the projective lengths (dimensions) of injective left $A$-modules, and the supremum of the injective lengths (dimensions) of projective left $A$-modules, respectively. The left finitistic dimension $\mathrm{fin.dim}(A)$ of $A$ is defined as the supremum of the projective dimensions of those left $A$-modules that have finite projective dimension. The global Gorenstein dimension of $A$ is denoted by $\mathrm{Ggldim}(A)$. The algebra $A$ is {\em Gorenstein} if  $\mathrm{Ggldim}(A)$ is finite.

We have the following invariants, which are preserved by stable equivalences of Morita type.

\begin{corollary}\label{cor:invAdjType}
Suppose that two bimodules  $_{B}M_{A}$ and $_{A}N_{B}$ define a stable equivalence of Morita type between algebras $A$ and $B$, such that both $(M\otimes_{A}-, N\otimes_{B}-)$ and $(N\otimes_{B}-, M\otimes_{A}-)$ are adjoint pairs.
\begin{enumerate}
\item For any $i\geq 0$, $K_i^G(A)\cong K_i^G(B)$.
\item $A$ is CM-free if and only if $B$ is CM-free. If $A$ and $B$ are Artin algebras, then $A$ is CM-finite if and only if $B$ is CM-finite.
\item $A$ is Gorenstein if and only if $B$ is Gorenstein. In this case, there are equalities
$$\mathrm{Ggldim}(A) = \mathrm{spli}(A) = \mathrm{silp}(A)=\mathrm{fin.dim}(A)\\
 = \mathrm{fin.dim}(B) = \mathrm{silp}(B) = \mathrm{spli}(B)  = \mathrm{Ggldim}(B).$$
\end{enumerate}
\end{corollary}

\begin{proof}
The statement (1) is from Proposition \ref{prop:StaK-group}, (2) is from Proposition \ref{prop:CM-inv}, and (3) follows by \cite[Theorem 4.1]{Emm12} together with  \cite[Corollary 3.3, Theorem A.6]{CR20}.
\end{proof}

\begin{remark}
Analogous to stable equivalence of Morita type, Chen and Sun introduced in their unpublished manuscript a notion of \emph{singular equivalence of Morita type}, by replacing the isomorphisms by $N\otimes_{B} M \cong A\oplus P'$ and $M\otimes_{A} N\cong B\oplus Q'$, where $_{A}P'_{A}$ and $_{B}Q'_{B}$ are bimodule with finite projective dimensions; see for example \cite[Definition 2.1]{ZZ13}. Under mild conditions, $(M\otimes_{A}-, N\otimes_{B}-)$ is also a Frobenius pair of faithful functors. Moreover, for any $A$-module $X$ and any $B$-module $Y$, the modules $\mathrm{Coker}(\eta_X)\cong P'\otimes_{A}X$ and $\mathrm{Ker}(\varepsilon_Y)\cong Q'\otimes_{B}Y$ are of finite projective dimension. Hence, the equivalent conditions in Theorem \ref{thm:FQE} are satisfied, and furthermore, the above results also hold.
\end{remark}

\section{\bf Examples}

In this section, we illustrate the aforementioned results by the following two examples.

\begin{example}\label{eg1}
Let $A$ and $B$ be algebras given by the following quivers with the relations, respectively:
\begin{center}$\xymatrix{ \bullet\ar@<0.5ex>[r]^{\alpha} &\bullet \ar@<0.5ex>[l]^{\beta} }$, $\beta\alpha\beta\alpha = 0$; \end{center}
\begin{center} $\xymatrix{ \bullet\ar@<0.5ex>[r]^{x} &\bullet \ar@<0.5ex>[l]^{y} \ar@(ur,dr)^{z} }$,  $yx = zx = yz = z^2 - xy = 0$. \end{center}

It follows from  \cite[Example, pp.583]{LX07} that there is a stable equivalence of adjoint type between $A$ and $B$. Note that $A$ is a Gorenstein algebra whose self-injective dimension on both sides is 2, then both $A$-mod and $B$-mod are Gorenstein categories. By Proposition \ref{prop:AdjType}, considering the Gorenstein projective model structures, there is a Quillen equivalence between $A$-mod and $B$-mod, which also induces an equivalence between the stable categories of finitely generated Gorenstein projective modules over $A$ and $B$. In the sense of Definition \ref{def:G-staFroF}, we say that there is a Gorenstein stable Frobenius pair between $A$ and $B$.
\end{example}

We refer to \cite{CSZ18} for the details on the notions of perfect pairs and perfect paths for monomial algebras. It follows from \cite[Theorem 4.1]{CSZ18} that $A\beta\alpha$ is the only finitely generated indecomposable non-projective Gorenstein projective $A$-module, which is corresponding to the unique perfect path $\beta\alpha$ in $A$. By \cite[Theorem 5.7]{CSZ18}, it follows that the stable category $\underline{\mathcal{GP}}(A)$ is equivalent to the semisimple triangulated category $\mathcal{T}_{1}$. We refer to \cite{CSZ18} for the definition of $\mathcal{T}_1$.

\begin{proposition}\label{prop:invEg1}
For the algebras $A$ and $B$ in Example \ref{eg1}, we have the following:
\begin{enumerate}
\item Both $A$ and $B$ are Gorenstein algebras, and moreover, $\mathrm{Ggldim}$, $\mathrm{spli}$, $\mathrm{silp}$ and $\mathrm{fin.dim}$ of both $A$ and $B$ are equal to 2.
\item Both $A$ and $B$ are CM-finite, and moreover, $\underline{\mathcal{GP}}(A)\simeq\mathcal{T}_{1} \simeq \underline{\mathcal{GP}}(B)$.
\item For any $i\geq 0$, $K_i^G(A)\cong K_i^G(B)$.
\end{enumerate}
\end{proposition}

The following implies that for the algebras $A$ and $B$ in Example \ref{eg1}, $K_0^G(A)\cong \mathbb{Z}\cong K_0^G(B)$.

\begin{proposition}\label{prop:calK0}
If there is only one finitely generated indecomposable non-projective Gorenstein projective module over $A$, then $K_0^G(A)\cong \mathbb{Z}$.
\end{proposition}

\begin{proof}
Assume that $G$ is the only finitely generated indecomposable non-projective Gorenstein projective $A$-module, then $\mathcal{GP}(A) = \mathrm{add}G$. Let $X$ be any finitely generated Gorenstein projective module, and assume that $X\oplus Y\cong G^n$ for some $A$-module $Y$. Since $G$ is indecomposable, $X$ is isomorphic to $G^m$ for some $m\leq n$.

It suffices to prove that the rank of $X$ is well-defined. Now assume that $X$ is weakly equivalent to both $G^m$ and $G^n$, where $m\leq n$. There is a cofibration sequence $G^m\rightarrowtail G^n\twoheadrightarrow G^{n-m}$. Since $G^m$ is weakly equivalent to $G^n$, $G^{n-m}$ should be a projective module. This is impossible by hypothesis if $m\neq n$. Hence, the rank induces an isomorphism  $K_0(A)\cong \mathbb{Z}$.
\end{proof}

\begin{proposition}\label{prop:calK1}
Let $A$ be the $k$-algebra in Example \ref{eg1}, $G = A\beta\alpha$. Then $\underline{\mathrm{End}}(G) \cong k$. Moreover, $K_1^G(B)\cong K_1^G(A) \cong k^{\times}$, where $k^{\times}$ is the group of units of $k$.
\end{proposition}

\begin{proof}
We have the following well-defined morphism of right $A$-modules
$$\theta: \beta\alpha A \longrightarrow (A\beta\alpha)^{*} = \mathrm{Hom}_{A}(A\beta\alpha, A),$$
which is given by $\theta(\beta\alpha x)(y) = yx$ for any $\beta\alpha x\in \beta\alpha A$ and $y\in A\beta\alpha$.

Let $f\in (A\beta\alpha)^{*} = \mathrm{Hom}_{A}(A\beta\alpha, A)$ be a morphism of left $A$-modules. Note that when $f\neq 0$, there are only two options to define it: $f(\beta\alpha) = \beta\alpha$, or $f(\beta\alpha) = \beta\alpha\beta$. Otherwise, let $f(\beta\alpha) = \beta$ for instance, then for $x=\beta\alpha$, we have $f(x\beta\alpha)=f(0)=0$ while $xf(\beta\alpha) = x\beta = \beta\alpha\beta$.

Now we have a map $\phi: (A\beta\alpha)^{*} = \mathrm{Hom}_{A}(A\beta\alpha, A)\rightarrow \beta\alpha A$, which is given by $\phi(f) = f(\beta\alpha)$. It is direct to check that $\phi\theta = \mathrm{Id}$ and $\theta\phi = \mathrm{Id}$. Then $\theta$ is an isomorphism. By \cite[Lemma 2.3]{CSZ18}, we have a $k$-linear isomorphism
$$\Lambda = \underline{\mathrm{End}}(G) \cong \frac{\beta\alpha A\cap A\beta\alpha}{\beta\alpha A\beta\alpha} \cong k.$$

For the field $k$, the determinant induces an isomorphism $\mathrm{det}: GL(k)/E(k)\rightarrow k^{\times}$; see for example \cite[Proposition 2.2.2]{Ros94}. In this case, the elementary matrix $E_{n}(k)$ is the matrix of determinant 1, and $E(k)$ is precise the special linear group of $k$. Hence, by Theorem \ref{thm:CharK1} we have $K_1^G(A)\cong k^{\times}$. This completes the proof.
\end{proof}

\begin{example}\label{eg2}
Let $A$ be the algebra given by the quiver
$$\xymatrix{
& 1 \ar[r]^{\alpha}  &2\ar[d]^{\beta}\\
& &3\ar[ul]^{\gamma}  }$$
with relations $\gamma\beta\alpha = \beta\alpha\gamma\beta = 0$. Let $B$ be the algebra given by the quiver
$$\xymatrix{ 1\ar@<0.5ex>[r]^{\rho} &2 \ar@<0.5ex>[l]^{\rho'}\ar@<0.5ex>[r]^{\delta} &3\ar@<0.5ex>[l]^{\delta'} }$$
with relations $\delta\rho = \rho'\rho = \rho'\delta' = \rho\rho'- \delta'\delta = 0$. By \cite[Example, pp.11]{Xi08}, there is a stable equivalence of adjoint type between $A$ and $B$. Then, there is a Quillen equivalence between $A$-mod and $B$-mod, which also induces an equivalence between the stable categories of finitely generated Gorenstein projective modules over $A$ and $B$.
\end{example}

\begin{proposition}\label{prop:invEg2}
For the algebras $A$ and $B$ in Example \ref{eg2}, we have the following:
\begin{enumerate}
\item Both $A$ and $B$ are CM-free, and moreover, $\underline{\mathcal{GP}}(A)\simeq \underline{\mathcal{GP}}(B)\simeq 0$.
\item For any $i\geq 0$, $K_i^G(A)\cong K_i^G(B) = 0$.
\end{enumerate}
\end{proposition}

\begin{proof}
Note that $A$ is a Nakayama monomial algebra. By the criterion in \cite[Lemma 6.1 (3)]{CSZ18}, we get all perfect pairs  $(\gamma, \beta\alpha)$, $(\gamma\beta, \alpha)$, $(\beta, \alpha\gamma\beta)$, $(\beta\alpha, \gamma\beta)$ and $(\beta\alpha\gamma, \beta)$. Then there is no perfect path in $A$. Hence, it follows from \cite[Corollary 4.2]{CSZ18} that $A$ is CM-free, that is, all Gorenstein projective $A$-modules are projective. This yields that $\underline{\mathcal{GP}}(A)\simeq 0$ and $K_i^G(A) = 0$. The assertions for $B$ hold since there is a Quillen equivalence between $A$-mod and $B$-mod.
\end{proof}

\begin{remark}
Since $A$ and $B$ in Example \ref{eg2} are of finite global dimension, Gorenstein projective modules over $A$ and $B$ coincide with projective modules over $A$ and $B$, respectively.
\end{remark}

\appendix

\section{\bf Waldhausen's construction}

For the clarity and completeness, in this section we recall some basic notions on simplicial sets, geometric realization of simplicial sets, and Waldhausen's $S_{\bullet}$-construction. For more details, we refer to \cite{Wal85, Web13} and other literatures on related topics.

Let $\Delta$ be the category whose objects are the finite ordered set $[n]=\{0,1,\cdots, n\}$ for integers $n\geq 0$, and whose morphisms are nondecreasing monomorphism functions, i.e. function $f: [n]\rightarrow [m]$ such that $i\leq j$ implies $f(i)\leq f(j)$ for $0\leq i,j\leq n$.

Recall that a simplicial set is a functor $\Delta^{op}\rightarrow \mathrm{Set}$. Alternatively, there is a combinatorial definition. A simplicial set consists of a sequence of sets $X_0, X_1,\cdots$ and for each $n\geq 0$, there are functions $d_i: X_n\rightarrow X_{n-1}$ and $s_i: X_n\rightarrow X_{n+1}$ for each $i$ with $0\leq i\leq n$ such that
\begin{subequations}
  \begin{align}
  &d_id_j = d_{j-1}d_i, ~~~~\mbox{ if } i <j,  \nonumber\\
  &s_is_j = s_{j+1}s_i, ~~~~\mbox{ if } i \leq j,  \nonumber\\
  & d_is_j= \left\{\begin{array}{ll}
~~s_{j-1}d_{i},  ~~~~\mbox{ if } i < j, \\[1\jot]
~~\mathrm{Id}, \quad\quad~~~~\mbox{ if } i = j, j+1,\\[1\jot]
~~s_jd_{i-1}, ~~~~\mbox{ if } i > j+1. \nonumber
\end{array} \right.
\end{align}
\end{subequations}
These equations can be illustrated in the category $\Delta$, where the face map $d_i$ is defined to be the map $\{0,1,\cdots, n\}\rightarrow \{0,\cdots,\hat{i},\cdots, n\}$ which deletes $i$, and the degeneracy  map
$s_i$ is defined to be the map $\{0,1,\cdots, n\}\rightarrow \{0,\cdots, i-1, i, i, i+1,\cdots, n\}$ which duplicates $i$.

We now outline how Waldhausen constructs a simplicial cateogory $S_{\bullet}\mathcal{C}$ from a Waldhausen category $\mathcal{C}$.

\begin{definition}\label{def:wSC}
Let $S_n\mathcal{C}$ be the category whose objects are sequences of $n$ cofibrations in $\mathcal{C}$
$$X_n:\quad 0 = C_0 \rightarrowtail C_1\rightarrowtail C_2\rightarrowtail \cdots \rightarrowtail C_n$$
together with a choice of every subquotient $C_{ij} = C_j/C_i$ for $0< i< j<n$. These choices are to be compatible in the sense that there is a commutative diagram
$$\xymatrix@C=20pt@R=15pt{
0\ar@{>->}[r] &C_1\ar@{>->}[r] & C_2 \ar@{>->}[r]\ar@{->>}[d] & C_3 \ar@{>->}[r]\ar@{->>}[d] &\cdots \ar@{>->}[r] & C_n \ar@{->>}[d]\\
& &C_{12}\ar@{>->}[r] & C_{13} \ar@{>->}[r]\ar@{->>}[d]  &\cdots \ar@{>->}[r] & C_{1n} \ar@{->>}[d]\\
& & &C_{23} \ar@{>->}[r] &\cdots \ar@{>->}[r] & C_{2n} \ar@{->>}[d]\\
& & & & &\cdots \ar@{->>}[d]\\
& & & & &C_{n-1,n}
}$$
It is convenient to let $C_i = C_{0i}$ and $C_{ii} = 0$ at times.

For $0< i\leq n$, the functor $d_i: S_n\mathcal{C}\rightarrow S_{n-1}\mathcal{C}$ is defined by omitting the row $C_{i*}$ and the column containing $C_i$, and reindexing the $C_{jk}$ as needed. The functor
$s_i: S_n\mathcal{C}\rightarrow S_{n+1}\mathcal{C}$ is defined by duplicating $C_i$ and reindexing with the normalization $C_{i,i+1} = 0$.

For two $n$-simplices $X_n, X_n^{'}\in S_n\mathcal{C}$, a weak equivalence is a map such that each $C_i\rightarrow C_i^{'}$ (hence, each $C_{ij}\rightarrow C_{ij}^{'}$) is a weak equivalence in $\mathcal{C}$. The map $C_i\rightarrow C_i^{'}$ is a cofibration when for every $0\leq i<j<k\leq n$ the map of cofibration sequences
$$\xymatrix@C=30pt@R=20pt{
C_{ij}\ar@{>->}[r] \ar[d] & C_{ik} \ar@{>->}[r]\ar[d]  & C_{jk} \ar[d]\\
C_{ij}^{'}\ar@{>->}[r]  & C_{ik}^{'} \ar@{>->}[r]  & C_{jk}^{'}
}$$
is a cofibration in $S_2\mathcal{C}$.

Then the $S_n\mathcal{C}$ fit together to form a simplicial Waldhausen category $S_{\bullet}\mathcal{C}$, and the subcategories $wS_{n}\mathcal{C}$ of weak equivalences fit together to form a simplicial category $wS_{\bullet}\mathcal{C}$.
\end{definition}

Now, we consider the geometric realization of the simplicial category $wS_{\bullet}\mathcal{C}$. Let $|\Delta^n|$ be the stand $n$-dimensional geometric simplex in $\mathbb{R}^{n+1}$, that is,
$$|\Delta^n| = \{ (r_0,\cdots,r_n): 0\leq r_i\leq 1, \sum r_i =1\}.$$
The coface map $D^i$ induces the inclusion $|\Delta^{n-1}|$ into $|\Delta^n|$ as the $i^{th}$ face, i.e. $D^i(r_0,\cdots,r_{n-1}) = (r_0,\cdots,r_{i-1}, 0, r_i, \cdots, r_{n-1})$. The codegeneracy map $S^i$ induces the projection
 $|\Delta^{n+1}|\rightarrow |\Delta^n|$ onto the $i^{th}$ face, i.e. $S^i(r_0,\cdots,r_{n+1}) = (r_0,\cdots,r_{i-1}, r_i + r_{i+1}, \cdots, r_{n+1})$.

\begin{definition}\label{def:GeoRea}
Let $X$ be a simplicial set. The realization $|X|$ of $X$ is given by $$|X| = (\coprod X_n \times |\Delta^n|)/\sim,$$
where $\sim$ is the equivalence relation generated by the relations $(x, D^i(p))\sim (d_i(x), p)$ for $x\in X_{n}$, $p\in |\Delta^{n-1}|$, and the relations $(x, S^i(p))\sim (s_i(x), p)$ for $x\in X_{n}$, $p\in |\Delta^{n+1}|$.
\end{definition}

In an intuitive sense, the equivalence relation collapses degeneracies and glues together faces. Say $x\in X_n$ an $n$-simplex in a simplicial set $X$. For each $n\geq 0$, topologize the product $X_n \times |\Delta^n|$ as the disjoint union of copies of the $n$-simplex $|\Delta^n|$ indexed by the elements $x$ of $X_n$. The first relation $(x, D^i(p))\sim (d_i(x), p)$ glues points $p\in |\Delta^{n-1}|$ on the boundary of $x$ to their corresponding location in the realization of $|\Delta^n|$ of $x\in X_{n}$, and thus does the work of gluing the faces of the realization of $x$ to the realizations of the faces of $x$.  The second relation  $(x, S^i(p))\sim (s_i(x), p)$ implies that a point $p$ in a realization $|\Delta^{n+1}|$ of $s_i(x)$ as ought to be collapsed via $S^i$ since $s_i(x)$ is degenerate for $x\in X_{n-1}$.

It is easy to see that in forming $|X|$ we can ignore every $n$-simplex of the form $s_i(x)\times |\Delta^n|$, and then the elements $s_i(x)$ are degenerate. The nondegenerate elements of $X_n$ index the $n$-cells of $|X|$, which implies that $|X|$ is a CW-complex.

Since $S_0\mathcal{C}$ is trivial, $|wS_{\bullet}\mathcal{C}|$ is a connected space.  The object of the category $S_1\mathcal{C}$ is of the form $0\rightarrowtail C_1$, and is thus isomorphic to $\mathcal{C}$. Hence $wS_1\mathcal{C}$ is isomorphic to $w\mathcal{C}$. The ``1-skeleton'' is obtained from the `0-skeleton'' by attaching of  $|wS_{1}\mathcal{C}|\times |\Delta^1|$. It results that the ``1-skeleton'' is naturally isomorphic to the suspension $\Sigma|w\mathcal{C}|$, and then $\Sigma|w\mathcal{C}|$ is a subspace of $|wS_{\bullet}\mathcal{C}|$. Therefore, there is an inclusion of  $|w\mathcal{C}|$ into the loop space $\Omega|wS_{\bullet}\mathcal{C}|$ of $|wS_{\bullet}\mathcal{C}|$.

We have the following definition from \cite[Definition IV 8.5]{Web13}.

\begin{definition}\label{def:K-group}
Let $\mathcal{C}$ be a small Walddhausen category. Its algebraic $K$-theory space is the loop space $\Omega|wS_{\bullet}\mathcal{C}|$. For  $i\geq 0$, the $K$-groups of $\mathcal{C}$ are defined to be the homotopy groups
$$K_i(C) = \pi_i(\Omega|wS_{\bullet}\mathcal{C}|) = \pi_{i+1}(|wS_{\bullet}\mathcal{C}|).$$
\end{definition}

It follows that each object of $\mathcal{C}$ yields an element of $\pi_{1}(|wS_{\bullet}\mathcal{C}|)$, and each weak equivalence in $\mathcal{C}$ yields an element of $\pi_{2}(|wS_{\bullet}\mathcal{C}|)$.

\section{\bf Quillen model categories revisited}
\label{sec:ModCat}

For the sake of convenience, we will briefly recall some notions and facts on model categories. We refer to \cite{DS95, Hov99, Qui67} for details.

\begin{definition}\label{def:MCat}
A {\em model structure} on a category $\mathcal{A}$ is a triple of classes of specified morphisms, called weak equivalences, cofibrations and fibrations, satisfying the following axioms:
\begin{enumerate}
\item If $f$ and $g$ are morphisms of $\mathcal{A}$ such that $gf$ is defined and if two of the three maps $f$, $g$ and $gf$ are weak equivalences, then so is the third.
\item If $f$ is a retract of $g$ and $g$ is a weak equivalence, cofibration, or fibration, then so is $f$.
\item In any commutative square
$$\xymatrix{
A\ar[d]_{i}\ar[r]^{f} &X\ar[d]^{p}\\
B\ar[r]^{g} &Y}$$
there exists a lift $h:B\rightarrow X$ such that $hi=f$ and $ph=g$ in either of the following two situations: (i) $i$ is a trivial cofibration (i.e. both a cofibration and a weak equivalence) and $p$ is a fibration, or (ii)
$i$ is a cofibration and $p$ is a trivial fibration (i.e. both a fibration and a weak equivalence).
\item Any morphism $f$ can be factored in two ways: (i) $f=pi$ where $i$ is a cofibration and $p$ is a trivial fibration, and (ii) $f=pi$ where $i$ is a trivial cofibration and $p$ is a fibration.
\end{enumerate}

A {\em model category} is a bicomplete category (i.e. a category possessing arbitrary small limits and colimits) equipped with a model structure.
\end{definition}

Given a model category $\mathcal{A}$, an object $X\in \mathcal{A}$ is called {\em trivial} if $0\rightarrow X$ (equivalently, $X\rightarrow 0$) is a weak equivalence. Similarly, it is called {\em cofibrant} if $0\rightarrow X$ is a cofibration, and it is called {\em fibrant} if $X\rightarrow 0$ is a fibration. The subcategories of trivial, cofibrant and fibrant objects will be denoted by $\mathcal{A}_{tri}$, $\mathcal{A}_{c}$ and $\mathcal{A}_{f}$, respectively.

We refer to \cite[Lemma 5.1, 5.3]{DS95} for the following.

\begin{definition}\label{def:QRrep}
Let $\mathcal{A}$ be a model category, and $X$ an object in $\mathcal{A}$. By factoring $0\rightarrow X$ as a cofibration $0\rightarrow Q(X)$ follows by a trivial fibration $Q(X)\rightarrow X$, we get a functor $X\rightarrow Q(X)$ such that $Q(X)$ is cofibrant. We refer to $Q$ as the {\em cofibrant replacement functor} of $\mathcal{A}$. Similarly,  one can factor the map $X\rightarrow 0$ to obtain a trivial cofibration $X\rightarrow R(X)$ with $R(X)$ fibrant, and moreover, there is a {\em fibrant replacement functor} $R$ together with a trivial cofibration $X\rightarrow R(X)$ for any object $X\in \mathcal{A}$.
\end{definition}

The following is immediately from \cite[Definition 1.3.1]{Hov99} and \cite[Lemma 1.3.4]{Hov99}.

\begin{definition}\label{def:QuiFun}
Suppose $\mathcal{A}$ and $\mathcal{B}$ are model categories, $F: \mathcal{A}\leftrightarrows \mathcal{B} :H$ is an adjoint pair of functors. The functor $F$ is said to be a left Quillen functor if $F$ preserves cofibrations and trivial cofibrations. Similarly, $H$ is said to be a right Quillen functor if $H$ preserves fibrations and trivial fibrations. We call  $(F, H)$ a \emph{Quillen adjunction} if $F$ is a left Quillen functor, or equivalently, $G$ is a right Quillen functor.
\end{definition}

For the bicomplete model category $\mathcal{A}$, the associated {\em homotopy category} ${\rm Ho}(\mathcal{A})$ is constructed by formally inverting the weak equivalences, i.e. localization with respect to weak equivalences.

\begin{definition}\label{def:derFun}
Let $F: \mathcal{A}\leftrightarrows \mathcal{B} :H$ be a Quillen adjunction between model categories $\mathcal{A}$ and $\mathcal{B}$. The total left derived functor
$LF$ is defined to be the composition
$$\xymatrix{{\rm Ho}(\mathcal{A})\ar[r]^{{\rm Ho}(Q)} & {\rm Ho}(\mathcal{A}_{c})\ar[r]^{{\rm Ho}(F)} & {\rm Ho}(\mathcal{B}).}$$
Similarly, the total right derived functor $RH$ is defined to be the composition
$$\xymatrix{{\rm Ho}(\mathcal{B})\ar[r]^{{\rm Ho}(R)} & {\rm Ho}(\mathcal{B}_{f})\ar[r]^{{\rm Ho}(H)} & {\rm Ho}(\mathcal{A}).}$$
\end{definition}

\begin{definition}\label{def:QuiEqu}
A Quillen adjunction $F: \mathcal{A}\leftrightarrows \mathcal{B} :H$ is called a Quillen equivalence if and only if for all cofibrant objects $X\in \mathcal{A}$ and fibrant objects $Y\in \mathcal{B}$, a map $F(X)\rightarrow Y$
is a weak equivalence in $\mathcal{B}$ if and only if the corresponding map $X\rightarrow H(Y)$ is a weak equivalence in $\mathcal{A}$.
\end{definition}

There are some useful criterions for checking when a given Quillen adjunction is a Quillen equivalence; see \cite[Proposition 1.3.13]{Hov99} and \cite[Corollary 1.3.16]{Hov99}. It is worth to note that $F: \mathcal{A}\leftrightarrows \mathcal{B} :H$ is a Quillen equivalence if and only if $LF: \mathrm{Ho}(\mathcal{A})\leftrightarrows \mathrm{Ho}(\mathcal{B}) :RH$ is an adjoint equivalence of homotopy categories;
see \cite[Proposition 1.3.13]{Hov99}.

\begin{ack*}
The author is grateful to the referee for helpful comments and suggestions which result in a significant improvement of the paper.
The author is greatly indebted to Professor Xi Changchang and Professor S. Estrada for helpful suggestions. The research work is supported by the National Natural Science Foundation of China (Grant No. 11871125).

After the paper was published, Professor Henrik Holm kindly points out the reference \cite{Hol15}, and give a specific comments on similarities and differences between Theorem \ref{thm:CharK1} and the interesting result in \cite[Theorem 2.12]{Hol15} for Gorenstein $K_1$ groups. The author is very grateful to him for the reference and helpful comments.
\end{ack*}


\bigskip
\begin{thebibliography}{99}

\bibitem{AB69} M. Auslander, M. Bridger, \emph{Stable module category}, Mem. Amer. Math. Soc. 94., 1969.

\bibitem{AR90} M. Auslander, I. Reiten, On a theorem of E. Green on the dual of the transpose, \emph{in
``Representations of finite dimensional algebras, Tsukuba, 1990''}, CMS Conf. Proc. 11, pp. 53--65, Amer. Math. Soc., Providence, RI, 1991.

\bibitem{Bel00} A. Beligiannis, The homological theory of contravariantly finite subcategories: Auslander-Buchweitz contexts, Gorenstein categories and (co)stabilization, \emph{Commun. Algebra} {\bf 28} (2000) 4547-4596.

\bibitem{Bel05} A. Beligiannis, Cohen-Macaulay modules, (co)torsion pairs and virtually Gorenstein algebras, \emph{J. Algebra} {\bf 288} (2005) 137-211.

\bibitem{BR07} A. Beligiannis, I. Reiten, \emph{Homological and Homotopical Aspects of Torsion Theories}, Mem. Amer. Math. Soc. \textbf{188} (883), 2007.

\bibitem{Bro92} M. Brou\'{e}, Equivalences of blocks of group algebras. In \emph{Finite-dimensional algebras and related topics} (Ottawa, ON, 1992), 1-26. NATO Adv. Sci. Inst. Ser. C Math. Phys. Sci. 424, Kluwer Acad. Publ., Dordrecht, 1994.

\bibitem{BGS08} K.A. Browna, I.G. Gordon, C.H. Stroppel, Cherednik, Hecke and quantum algebras as free Frobenius and Calabi-Yau extensions,  \emph{J. Algebra} \textbf{319}(3) (2008) 1007-1034.

\bibitem{CIGTN99} F. Casta\~{n}o Iglesias, J. G\'{o}mez Torrecillas, C. N\u{a}st\u{a}sescu, Frobenius functors: Applications, \emph{Commun. Algebra} \textbf{27}(10) (1999)
 4879--4900.

\bibitem{Chen11} X.-W. Chen, Relative singularity categories and Gorenstein-projective modules, \emph{Math. Nachr.} \textbf{284}(no. 2-3) (2011) 199-212.

\bibitem{CR20} X.-W. Chen, W. Ren, Frobenius functors and Gorenstein homological properties, \emph{J. Algebra} {\bf 610} (2022) 18-37.

\bibitem{CSZ18} X.-W. Chen, D. Shen, G. Zhou, The Gorenstein-projective modules over a monomial algebra, \emph{Proc. Royal Soc. Edinbergh} \textbf{148A} (2018) 1115-1134.

\bibitem{CZ07} X.-W. Chen, P. Zhang, Quotient triangulated categories, \emph{Manuscripta Math.}  {\bf 123} (2007) 167-183.

\bibitem{DEH18} G. Dalezios, S. Estrada, H. Holm, Quillen equivalences for stable categories, \emph{J. Algebra} \textbf{501} (2018) 130-149.

\bibitem{DM07} A.S. Dugas, R. Mart\'{\i}nez-Villa, A note on stable equivalences of Morita type, \emph{J. Pure Appl. Algebra} \textbf{208} (2007) 421-433.

\bibitem{DS04} D. Dugger, B. Shipley, $K$-theory and derived equivalences, \emph{Duke Math. J.} \textbf{124} (3) (2004) 587-617.

\bibitem{DS09} D. Dugger, B. Shipley, A curious example of triangulated-equivalent model categories which are not Quillen equivalent, \emph{Algebr. Geom. Topol.} \textbf{9} (2009) 135-166.

\bibitem{DS95} W.G. Dwyer, J. Spalinski, Homotopy Theories and Model Categories. Handbook of algebraic topology (Amsterdam), North-Holland, Amsterdam, pp 73-126 (1995).

\bibitem{Emm12} I. Emmanouil, On the finiteness of Gorenstein homological dimensions, \emph{J. Algebra} \textbf{372} (2012) 376-396.

\bibitem{EEGR} E.E. Enochs, S. Estrada, J.R. Garc\'{\i}a-Rozas, Gorenstein categories and Tate cohomology on projective schemes, \emph{Math. Nachr.} {\bf 281}(4) (2008) 525-540.

\bibitem{EJ00}  E.E. Enochs, O.M.G. Jenda, \emph{Relative Homological Algebra}, De Gruyter Expositions in Mathematics no. 30, New York: Walter De Gruyter, 2000.

\bibitem{Gil11} J. Gillespie, Model structures on exact categories, \emph{J. Pure Appl. Algebra} \textbf{215}(12) (2011) 2892-2902.

\bibitem{Hol04}  H. Holm, Gorenstein homological dimensions, \emph{J. Pure Appl. Algebra} {\bf 189} (2004) 167-193.

\bibitem{Hol15} H. Holm, K-groups for rings of finite Cohen–Macaulay type, \emph{Forum Math.} {\bf 27} (2015) 2413–2452.

\bibitem{Hov99}  M. Hovey, \textit{Model categories}, Mathematical Surveys and Monographs vol. 63, American Mathematical Society, 1999.

\bibitem{Hov02}  M. Hovey, Cotorsion pairs, model category structures, and representation theory, \emph{Math. Z.} \textbf{241} (2002) 553-592.

\bibitem{HX18} W. Hu, C.C. Xi, Derived equivalences and stable equivalences of Morita type II, \emph{Rev. Mat. Iberoam.} \textbf{34}(1) (2018) 59-110.

\bibitem{Kad99} L. Kadison, \emph{New Examples of Frobenius Extensions}, Univ. Lecture Ser., vol. 14, Amer. Math. Soc., Providence, RI, 1999.

\bibitem{Kas61} F. Kasch, Projektive Frobenius-Erweiterungen, S.-B. Heidelberger Akad. Wiss. Math.-Nat. Kl. (1960/1961) 87-109.

\bibitem{Koc04} J. Kock, \emph{Frobenius Algebras and 2D Topological Quantum Field Theories}, London Math. Soc. Student Texts, vol.59, Cambridge University Press, 2004.

\bibitem{Liu03} Y.M. Liu, On stable equivalences of Morita type for finite dimensional algebras, \emph{Proc. Amer. Math. Soc.} \textbf{131} (2003) 2657-2662.

\bibitem{LX07} Y.M. Liu, C.C. Xi, Construction of stable equivalences of Morita type for finite dimensional algebras III, \emph{J. London Math. Soc.} \textbf{76}(2) (2007) 567-585.

\bibitem{Mor65}  K. Morita,  Adjoint pairs of functors and Frobenius extensions, \emph{Sci. Rep. Tokyo Kyoiku Daigaku, Sect. A} \textbf{9} (1965) 40-71.

\bibitem{Qui67} D.G. Quillen, Homotopical Algebra. Lecture Notes in Mathematics no. 43, Springer-Verlag, 1967.

\bibitem{Ren18} W. Ren, Gorenstein projective modules and Frobenius extensions, \emph{Sci. China Math.} \textbf{61}(7) (2018) 1175-1186.

\bibitem{Ric91} J. Rickard, Derived equivalences as derived functors, \emph{J. London Math. Soc.} \textbf{43} (1991) 37-48.

\bibitem{Rin13} C.M. Ringel, The Gorenstein-projective modules for the Nakayama algebras I, \emph{J. Algebra} \textbf{385} (2013) 241-261.

\bibitem{Ros94} J. Rosenberg, \emph{Algebraic K-theory and Its Applications}, Graduate Texts in Mathematics vol.147, Springer-Verlag, New York, 1994.

\bibitem{Sch02} M. Schlichting, A note on $K$-theory and triangulated categories, \emph{Invent. Math.} \textbf{150} (2002) 111--116.

\bibitem{Wal85} F. Waldhausen, Algebraic $K$-theory of spaces, \emph{Algebraic and geometric topology, Lecture Notes in Math. vol.1126}, pp. 318--419. Springer, Berlin, 1985.

\bibitem{Web13} C.A. Weibel, \emph{The K-book: An Introduction to Algebraic K-theory}, Graduate Studies in Mathematics vol. 145, Amer. Math. Soc., 2013.

\bibitem{Xi02} C.C. Xi, Representation dimension and quasi-hereditary algebras, \emph{Adv. Math.} \textbf{168} (2002) 193-212.

\bibitem{Xi08} C.C. Xi, Stable equivalences of adjoint type, \emph{Forum Math.} \textbf{20} (1) (2008) 81-97.

\bibitem{Xi21} C.C. Xi, Frobenius bimodules and flat-dominant dimensions, \emph{Sci. China Math.} \textbf{64}(1) (2021)  33-44

\bibitem{ZZ13} G.D. Zhou,  A. Zimmermann, On singular equivalences of Morita type, \emph{J. Algebra} {\bf 385} (2013) 64-79.

\end{thebibliography}
\end{document}